\documentclass[11pt,a4paper]{article}
\usepackage[margin=27mm]{geometry}
\usepackage[T1]{fontenc}
\usepackage{lmodern,microtype}
\usepackage{amsmath,amssymb,amsthm,array,needspace,graphicx,booktabs,float}
\usepackage[dvipsnames]{xcolor}
\usepackage[colorlinks=true,hyperfootnotes=false,linkcolor=MidnightBlue,citecolor=MidnightBlue,urlcolor=MidnightBlue]{hyperref}
\usepackage{fancyhdr}
\newtheorem{theorem}{Theorem}
\newtheorem{lemma}{Lemma}
\newtheorem{proposition}{Proposition}
\newtheorem{corollary}{Corollary}
\theoremstyle{definition}
\newtheorem{definition}{Definition}
\numberwithin{equation}{section}
\DeclareMathOperator{\Area}{Area}
\DeclareMathOperator{\diam}{diam}
\DeclareMathOperator{\arcosh}{arcosh}
\newcommand{\HH}{\mathbb H^2}
\newcommand{\RR}{\mathbb R}
\newcommand{\ZZ}{\mathbb Z}

\title{Area and diameter gaps for hyperbolic monotiles}
\author{Yixi Liao and Erxiao Wang\\[4pt]
\small Zhejiang Normal University}
\date{}
\hypersetup{pdftitle={Area and diameter gaps for hyperbolic monotiles},
pdfauthor={Yixi Liao and Erxiao Wang}}
\makeatletter
\renewcommand{\@maketitle}{\begin{center}
{\Large\@title\par}\vskip0.8em
{\normalsize\begin{tabular}[t]{c}\@author\end{tabular}\par}
\end{center}\vskip0.6em}
\makeatother

\graphicspath{{./}}
\newcommand{\dclosed}{\delta_n^{\mathrm{cl}}}
\newcommand{\dplane}{d_n^{\mathrm{pl}}}
\newcommand{\anorm}{a/\pi}
\newcommand{\ang}[1]{[#1]}
\newcommand{\ballarea}{\mathcal A}
\newcommand{\vertexset}{\mathcal V}
\newcommand{\starcount}{v}
\newcommand{\halfcount}[1]{v_{\mathrm{half},#1}}
\begin{document}
\maketitle
\begingroup
\renewcommand{\thefootnote}{}
\footnotetext{Corresponding author: Erxiao Wang
(\href{mailto:wang.eric@zjnu.edu.cn}{wang.eric@zjnu.edu.cn}).
Research was supported by National Natural Science Foundation of China
NSFC-RGC 12361161603 and Key projects of Zhejiang Natural Science
Foundation No.\ LZ22A010003.}
\endgroup
\begin{abstract}
For each fixed $n$, we prove a positive lower bound for the diameter
of a compact simple geodesic $n$-gonal monotile of the hyperbolic
plane. We also prove a positive lower bound for the area of such a
tile in a finite monohedral tiling of a closed hyperbolic surface,
independent of the topology and metric. Both bounds become independent
of $n$ for vertex-proper tilings, in which every genuine tile vertex
belongs to at least three distinct tiles. This includes tilings by
convex polygons. Reflex angles and non-edge-to-edge incidences are
allowed. After qualitative proofs, we obtain explicit constants from
a sharp gap estimate for packing polytopes with arbitrarily coupled
nonnegative integer constraints. Exact corner balance applies on
closed surfaces; in the plane, covering duality and ball counts give
a boundary factor depending on the tile diameter. Following Zare,
we give, for each integer $q\ge2$, a geodesic construction with
$2q+3$ sides and diameter less than $3/q$, showing
that side counts cannot be unrestricted without an additional
condition. The closed-surface results extend to regular curved sides.
The geometric constants are effective but are not claimed to be optimal.
\end{abstract}

\section{Introduction}

How small can a polygonal monotile of the hyperbolic plane be?
We prove diameter gaps for plane monotiles and area gaps for finite
tilings of closed hyperbolic surfaces. A fixed number of geodesic
sides suffices for both conclusions. A condition on the number of
tiles meeting at vertices makes the bounds independent of the number
of sides, and includes all convex geodesic monotiles.

A tiling is \emph{monohedral} if its tiles are congruent, and
\emph{isohedral} if its symmetry group is transitive on tiles.
In curvature $-1$, Siegel's theorem~\cite{siegel} gives the sharp
area bound $\pi/42$ for isohedral plane monotiles, attained by the
$(2,3,7)$ reflection triangle in Figure~\ref{fig:intro-tilings}(a).
There is no corresponding area bound for arbitrary plane monotiles:
Di Giosia et al.~\cite[arXiv v1, Propositions 4.8--4.9]{digiosia}
construct triangles of every area $0<a<\pi$, and Margulis and
Mozes~\cite[Theorem 5]{margulis} construct strictly convex $n$-gonal
monotiles of every area $0<a<(n-2)\pi$ for $n\ge5$, as reported
in~\cite[arXiv v1, discussion following Corollary 4.10]{digiosia}.

Small area need not imply small diameter. Our plane theorem rules
out diameters tending to zero for all convex monotiles, even with
varying side counts. Zare~\cite{zare2010,zare} described how
nonconvex monotiles can have arbitrarily small diameter. Following
his idea, we give explicit geodesic $(2q+3)$-gons with area less
than $1/q^2$ and diameter less than $3/q$. Their side counts tend to
infinity; Figure~\ref{fig:intro-tilings}(b) shows the case $q=2$.
These examples explain the two hypotheses in our plane theorem:
fixed side count, or the vertex condition defined below.

\begin{figure}[H]
\centering
\begin{minipage}[t]{0.485\linewidth}
\centering
\includegraphics[width=0.85\linewidth]{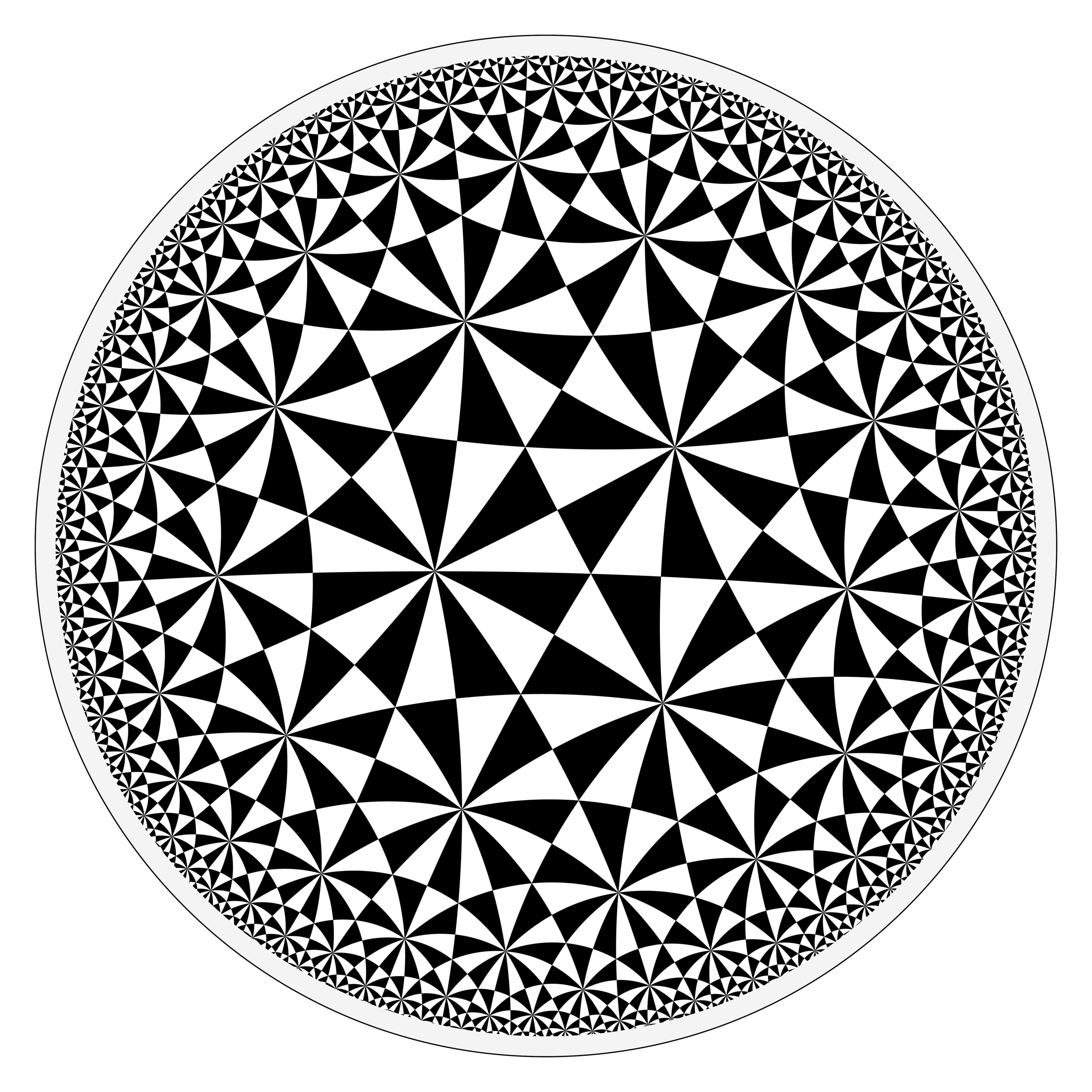}
\par\smallskip\small (a) Coxeter $(2,3,7)$ reflection tiling
\end{minipage}\hfill
\begin{minipage}[t]{0.485\linewidth}
\centering
\includegraphics[width=0.85\linewidth]{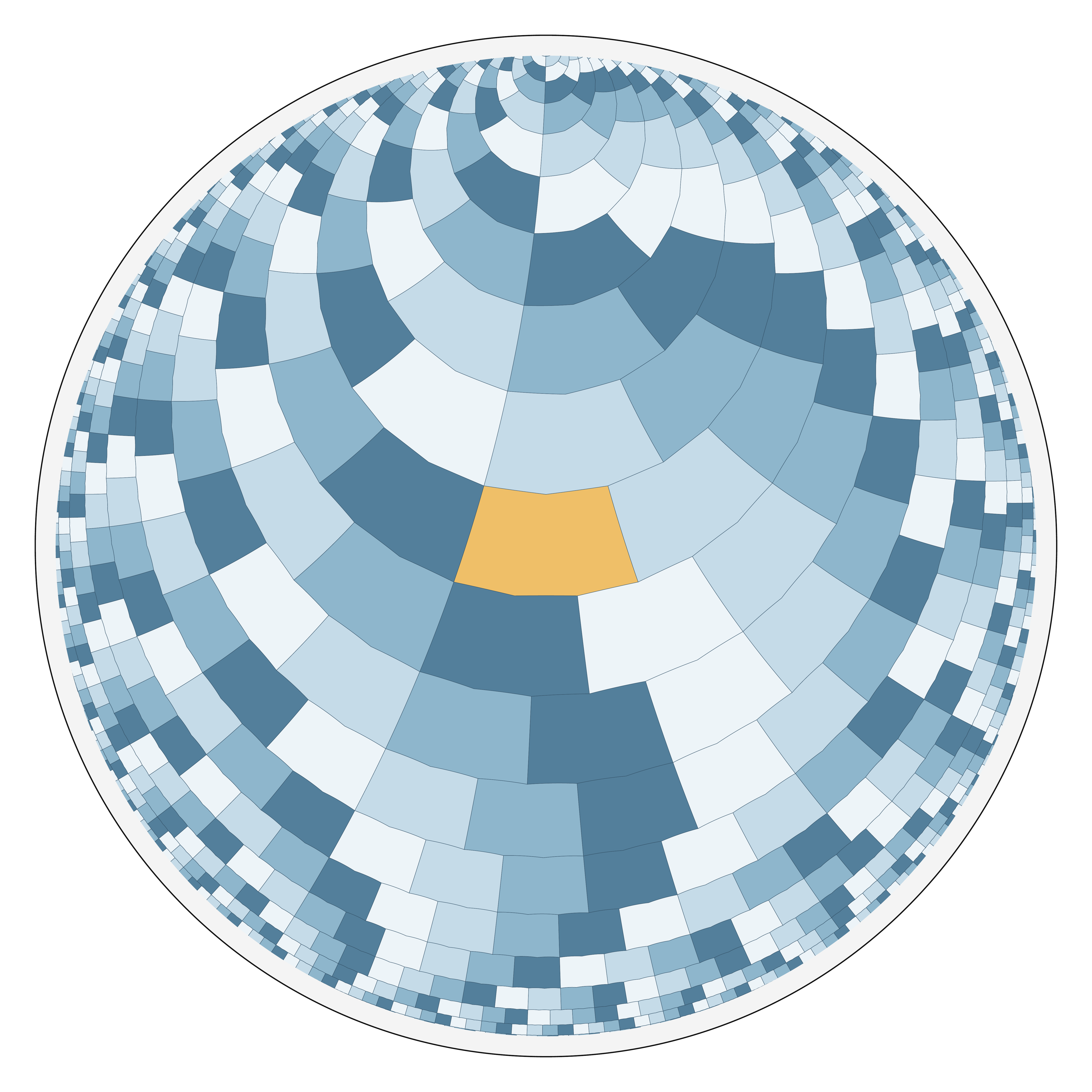}
\par\smallskip\small (b) A geodesic realisation of Zare's idea
\end{minipage}
\caption{Tilings in the Poincar\'e disc, shown for $|z|\le0.96$.
Left: the $(2,3,7)$ reflection tiling by triangles of area $\pi/42$.
Right: the nonconvex heptagonal member $q=2$ of the geodesic
realisation of Zare's idea in Section~\ref{sec:small};
the tile $T_{0,0}$ is shown in gold.}
\label{fig:intro-tilings}
\end{figure}

\subsection{The two main theorems}

All hyperbolic metrics have curvature $-1$.
A polygonal tile is a compact embedded closed disc with nonempty
interior and finitely many geodesic sides. Tiles cover the ambient
surface and have pairwise disjoint interiors. We remove redundant
straight corners, so $n$ counts genuine sides. Non-edge-to-edge
incidences are allowed. For tiles on a closed surface, congruence and convexity refer
to their lifts to $\HH$; $\diam_S$ denotes diameter in the
ambient surface metric. Congruence may reverse orientation; on a closed surface it need not come from a global surface isometry.

The \emph{tiling vertices} are the genuine corners of the tiles.
They include vertices at which only two tiles meet, but exclude
auxiliary subdivision points. Write $\mathcal V$ for their set.

\begin{definition}[Vertex-properness]\label{def:vertex-proper}
A tiling $\mathcal T$ is \emph{vertex-proper} if
\begin{equation}\tag{V}\label{eq:vertex-condition}
 \#\{T\in\mathcal T:w\in T\}\ge3
 \qquad\text{for every }w\in\mathcal V,
\end{equation}
where distinct tiles are counted. A prototile $P$ is
\emph{vertex-proper for $M$} if it admits at least one vertex-proper
monohedral tiling of the specified ambient surface $M$.
It is \emph{universally vertex-proper for $M$} if it admits a
monohedral tiling of $M$ and every such tiling is vertex-proper.
\end{definition}

Here $M$ is $\HH$ or a closed hyperbolic surface.
Condition~\eqref{eq:vertex-condition} allows $T$-junctions: a tile
whose side interior contains $w$ is counted even though $w$ is not
its corner. Every tiling by convex geodesic polygons is vertex-proper:
all sectors have angle at most $\pi$, and a genuine corner has
angle less than $\pi$, so two sectors cannot fill the tangent plane.
Thus a convex geodesic monotile is universally
vertex-proper, although our plane theorem only needs the existential
condition. Note that the authors' two series of works on spherical tilings (\cite{wangyan,liaowang}) 
provide many nonconvex vertex-proper monotiles. 

\begin{theorem}[Diameter bounds in the hyperbolic plane]
\label{thm:convex-diameter}
Let $P$ be a compact simple geodesic polygon admitting a monohedral
tiling of $\HH$, and put $D=\diam(P)$.
\begin{enumerate}
\item For each fixed $n\ge3$, there is a constant
$d_n^{\mathrm{pl}}>0$ such that $D>d_n^{\mathrm{pl}}$ whenever $P$
has $n$ sides. Reflex angles are allowed.
\item There is an absolute constant $d_*>0$ such that $D>d_*$
whenever $P$ is a vertex-proper prototile for $\HH$, independently
of the number of sides. In particular, this holds whenever $P$ is convex.
\end{enumerate}
Both conclusions allow non-edge-to-edge tilings and require no
periodicity or transitive symmetry group.
\end{theorem}

\begin{theorem}[Area bounds on closed hyperbolic surfaces]
\label{thm:gap}
Let a closed hyperbolic surface be tiled by finitely many
congruent geodesic polygons, each an embedded closed disc, and let
$a$ be their common area.
\begin{enumerate}
\item For each fixed $n\ge3$, there is a constant $c_n>0$ such that
$a\ge c_n$ whenever the tiles have $n$ sides. Reflex angles are allowed.
\item There is an absolute constant $c_*>0$ such that $a\ge c_*$
whenever the tiling is vertex-proper, independently of the number
of sides. In particular, this holds for convex tiles.
\end{enumerate}
The constants are independent of the topology and hyperbolic metric,
and both conclusions allow non-edge-to-edge tilings.
\end{theorem}

We first give qualitative proofs and then obtain explicit constants.
No optimality is asserted for these geometric constants. The
closed-surface area bounds also give ambient diameter bounds and
upper bounds for the number of tiles. In the plane, the convex
conclusion includes compact convex monotiles without any initial
assumption on their boundary, by Proposition~\ref{prop:convex-body-monotile}.

\subsection{Notation and proof mechanism}

The prototile is $P$, and its copies are $T_p$. For a finite tiling
of a closed surface $S$, let $v,e,f$ count vertices, edges and tiles
after subdivision at all genuine corners. Put $\chi=\chi(S)<0$;
then $v-e+f=\chi$ and
\begin{equation}\label{eq:surface-area}
 \Area(S)=-2\pi\chi.
\end{equation}
No orientability is assumed. For orientable genus $g$, $\chi=2-2g$;
for nonorientable genus $h$, $\chi=2-h$.
Label the corners cyclically by $i=1,\ldots,n$, with indices modulo
$n$. Write $\ang{i}$ for angle $i$, $x_i=\ang{i}/\pi$, and
$\bar i$ for the edge from $i$ to $i+1$.
A \emph{full vertex} is a corner of every incident tile.
A \emph{half vertex} is in the side interior of one tile and is a
corner of the others. Its genuine corner angles sum to $\pi$,
whereas those at a full vertex sum to $2\pi$.
Write $a=\Area(P)$, $D=\diam(P)$ in the plane, and
$\diam_S(T_p)$ for ambient tile diameters on $S$.

Tao's discussion of small spherical tiles~\cite{Tao} uses ball
counts, integer angle relations and determinants. Negative curvature
gives a different direction for the angle sum:
\[
 \sum_i x_i=n-2-\frac a\pi<n-2.
\]
Our integer-gap estimate controls packing maxima below integers.
On a closed surface, exact counts of each labelled corner identify the angle
sum with the maximum $\mu$ of a packing problem, so
$a/\pi=(n-2)-\mu$. The fixed-side qualitative proof takes limits
of angle vectors; it does not require compactness of polygonal shapes.

The quantitative tool is Lemma~\ref{lem:gap}. If
$Q=\{y\in\RR_{\ge0}^k:Ty\le b\}$ is bounded, $T$ has
nonnegative integer entries, $b$ has entries $1$ or $2$, and
$\mu=\max_Q\sum_i y_i<\ell$ for a positive integer $\ell$, then
\[
 \ell-\mu\ge
 \begin{cases}
 B_{k-\lfloor\ell/2\rfloor}^{-1},&1\le\ell\le2k,\\
 \ell-2k,&\ell>2k,
 \end{cases}
 \qquad B_0=1,\quad B_{j+1}=B_j(2B_j+1).
\]
The estimate is sharp for this class of integer systems, with
equality attained by diagonal constraints. For diagonal systems,
the same extremal gap follows from Ambro and
Barc\u{a}u~\cite[Theorem 0.2(a)]{ambrobarcau}; see also
Nathanson~\cite[Theorem 5]{nathanson} for greedy unit-fraction
underapproximations. We prove that the bound persists for arbitrary
coupled constraints, with no bound on their number or coefficients.
The arithmetic equality examples need not be realised by tilings.

In the plane, no finite total corner count is available. Let $\mu$
instead be the minimum of the associated covering problem.
Complementary slackness produces a packing problem with the same
optimum, so the integer estimate still applies. Complete vertex
stars counted in large balls give
\[
 e^{-2D}\left(n-2-\frac a\pi\right)
 \le\mu\le n-2-\varepsilon_n,
 \qquad \varepsilon_n=B_{\lceil n/2\rceil+1}^{-1}.
\]
The boundary factor does not disappear as the ball radius grows.
Together with $a\le2\pi(\cosh D-1)$, it excludes $D\to0$.
The qualitative proof uses balls whose radii tend to zero more
slowly than the tile diameters. It controls unbounded corner
multiplicities when angles tend to zero and uses finite convex
representations to recover exact angle equations. Vertex-properness
then supplies the Euler and adjacency-growth estimates needed to
remove the restriction on the side count.

\subsection{Related work and organisation}

The large-side Euler counts for edge-to-edge hyperbolic surface
tilings occur in Li, Wang, Wu and Yan~\cite{lisingle,limultiple,limax}.
Our fixed-side arguments also treat the small side counts and allow
reflex angles and half vertices. Di Giosia et al.~\cite{digiosia}
and Hirsch, Li, Petty and Xue~\cite{hirsch} study perimeter
minimisation, while Goodman-Strauss~\cite{goodman} gives criteria
for generic triangular hyperbolic plane tiles. Our large-side estimate uses
an elementary planar edge count. Related curvature and growth
estimates appear in~\cite{bauespeyerimhoff,keller}.

Section~\ref{sec:qualitative} gives the qualitative arguments,
and Section~\ref{sec:integer} proves the sharp integer estimate.
Sections~\ref{sec:convex-diameter} and~\ref{sec:gap} derive explicit
plane diameter and closed-surface area bounds.
Section~\ref{sec:classical-diameters} examines classical convex
examples, and Section~\ref{sec:small} gives the geodesic
realisation of Zare's idea.
Section~\ref{sec:extensions} treats regular curved sides, symmetry
and compact quotients.

\section{Qualitative proofs}\label{sec:qualitative}

We first prove the existence of fixed-side gaps by taking limits of
angle vectors. Both the plane and the closed-surface arguments allow
reflex angles. We then explain how vertex-properness gives uniform
bounds and give a qualitative form of the common integer-gap principle.
The explicit estimates are proved in Sections~\ref{sec:convex-diameter}
and~\ref{sec:gap}.

\subsection{Local finiteness and hyperbolic balls}\label{sec:local-finite}

Write
\[
 \ballarea(r)=2\pi(\cosh r-1)
\]
for the area of a hyperbolic ball of radius $r$.
A tiling by copies of a compact tile with positive area $a$, boundary
of area zero, and diameter $D$ is locally finite. Indeed, every tile meeting a compact set $K$
lies in its closed $D$-neighbourhood $K_D$, which has finite area.
Any $m$ such tiles have disjoint interiors and therefore satisfy
$ma\le\Area(K_D)$. Thus there are only finitely many of them.
The argument applies to each tiling separately and requires no common
inradius or area bound for a family of prototiles.

\subsection{A qualitative diameter gap for plane tiles}
\label{sec:qualitative-plane}

We give a qualitative proof of the fixed-side diameter bound for
simple geodesic plane monotiles, allowing reflex angles.
Proposition~\ref{prop:plane-fixed-n} gives an explicit bound.

\subsubsection{A finite representation of an integrable barycenter}

We use the following elementary finite-dimensional fact. The conclusion
concerns the convex hull itself, not merely its closure. See also
Bayer and Teichmann \cite[Theorem 1 and Remark 2]{bayerteichmann}.

\begin{lemma}\label{lem:plane-barycenter}
Let $(x_t)$ be a finite or countable family in $\RR^d$, and suppose
$p_t>0$, $\sum_t p_t=1$, and $\sum_t p_t\|x_t\|<\infty$.
Then $x=\sum_t p_t x_t$ is a convex combination of finitely many of
the points $x_t$.
\end{lemma}
\begin{proof}
Put $C=\operatorname{conv}\{x_t\}$ and
$H=\operatorname{aff}\{x_t\}$. Moving the unused mass of a partial
sum to one fixed support point gives finite convex combinations
converging to $x$, so $x\in\overline C$.
If $x$ were on the relative boundary of $\overline C$ in $H$, a
supporting affine function $\ell$, not identically zero on $H$, would
satisfy $\ell\ge0$ on $C$ and $\ell(x)=0$.
Integrability gives $\sum_t p_t\ell(x_t)=0$.
Thus $\ell(x_t)=0$ for every $t$, contrary to the definition of $H$.
Consequently $x$ is in the relative interior of $\overline C$.
Choose in $H$ a small simplex surrounding $x$, with vertices in
$\overline C$. Approximate its vertices by points of $C$ closely
enough that the new simplex still contains $x$. Convexity gives
$x\in C$, which is exactly the asserted finite representation.
The zero-dimensional case is immediate.
\end{proof}

\subsubsection{The gap for a fixed number of sides}

\begin{proposition}\label{prop:plane-fixed-n-qualitative}
For each fixed integer $n\ge3$ there is $d_n>0$ such that every
compact simple geodesic $n$-gonal monotile of $\HH$ has diameter at
least $d_n$. No lower bound on angles, side lengths, or aspect ratio
is assumed.
\end{proposition}
\begin{proof}
Suppose, to the contrary, that prototiles $P_j$ have diameters
$D_j\to0$ and areas $a_j>0$. Label their angles
$\ang{1}_j,\ldots,\ang{n}_j$, keeping equal angles separately. Write
$\boldsymbol\alpha^{(j)}=(\ang{1}_j,\ldots,\ang{n}_j)$
and $\boldsymbol\alpha^*=(\ang{1}^*,\ldots,\ang{n}^*)$
for the corresponding angle vectors.
For each tile copy choose one isometry from $P_j$ to that copy and
transport the labels by it. No consistency between these choices is
needed: every tile contributes exactly one occurrence of each label.
Since $a_j\le \ballarea(D_j)$, Gauss--Bonnet and a subsequence give
\begin{equation}\label{eq:plane-angle-limits}
 \sum_{i=1}^n\ang{i}_j=(n-2)\pi-a_j,
 \qquad \ang{i}_j\to\ang{i}^*\in[0,2\pi],
 \qquad \sum_{i=1}^n\ang{i}^*=(n-2)\pi.
\end{equation}

\smallskip\noindent\emph{Corner counts.}
Fix a centre in $\HH$, put $R_j=\sqrt{D_j}$ and
$\Lambda_j=\ballarea(R_j)/a_j$, and omit finitely many terms so that $D_j<R_j$.
Let $c_{i,j}$ count all occurrences of the $i$th labelled corner in
the closed ball $B(R_j)$, with incidences counted separately.
If that corner of a tile lies in $B(R_j)$, the whole tile lies in
$B(R_j+D_j)$. Conversely, if a tile meets $B(R_j-D_j)$, its
$i$th corner lies in $B(R_j)$. Hence
\begin{equation}\label{eq:plane-ball-count}
 \frac{\ballarea(R_j-D_j)}{a_j}\le c_{i,j}
 \le\frac{\ballarea(R_j+D_j)}{a_j}.
\end{equation}
Since $R_j\to0$ and $D_j/R_j\to0$, division by $\Lambda_j$ yields
\begin{equation}\label{eq:plane-label-balance}
 \frac{c_{i,j}}{\Lambda_j}\longrightarrow1\qquad(1\le i\le n).
\end{equation}

\smallskip\noindent\emph{Vertex types.}
At every point at which some tile has a genuine corner, take the
complete star of incident tiles, even when counting only star centres
in $B(R_j)$. Its type is
$t=(q_t,h_t)$, where
$q_t=(q_{t,1},\ldots,q_{t,n})\in\ZZ_{\ge0}^n\setminus\{0\}$
counts labelled genuine corners and $h_t\in\{0,1\}$ counts flat
incidences from side interiors. There is at most one flat incidence
because a genuine corner is present and the total angle is $2\pi$.
Set $r_t=2-h_t\in\{1,2\}$. Every type occurring in the $j$th tiling
satisfies the exact equation
\begin{equation}\label{eq:plane-star}
 q_t\cdot\boldsymbol\alpha^{(j)}=r_t\pi.
\end{equation}
Let $\starcount_{j,t}$ count star centres of type $t$ in $B(R_j)$ and set
$\lambda_{j,t}=\starcount_{j,t}/\Lambda_j$, assigning zero to types not occurring.
Local finiteness makes these counts finite. We have
\begin{equation}\label{eq:plane-star-moments}
 \sum_t\lambda_{j,t}q_{t,i}=\frac{c_{i,j}}{\Lambda_j}\to1,
 \qquad
 \sum_t\lambda_{j,t}\le\sum_{i=1}^n\frac{c_{i,j}}{\Lambda_j}=n+o(1).
\end{equation}
The set of all types is countable. A diagonal subsequence therefore
has $\lambda_{j,t}\to\lambda_t\ge0$ for every fixed $t$.

\smallskip\noindent\emph{Limiting moments.}
Put $I=\{i:\ang{i}^*>0\}$ and $Z=\{i:\ang{i}^*=0\}$.
The set $I$ is nonempty by \eqref{eq:plane-angle-limits}.
For $i\in I$, equation \eqref{eq:plane-star} gives a uniform bound
$q_{t,i}\le C_i$ for all types occurring at sufficiently large $j$;
one may take $C_i=\lceil4\pi/\ang{i}^*\rceil$.
Write $|q_{t,Z}|=\sum_{i\in Z}q_{t,i}$. For every $K>0$,
\begin{equation}\label{eq:plane-tail}
 \sum_{|q_{t,Z}|>K}\lambda_{j,t}
 \le\frac1K\sum_{i\in Z}\sum_t\lambda_{j,t}q_{t,i}
 =\frac{|Z|+o(1)}K.
\end{equation}
The contribution of these types to an $i\in I$ moment is at most
$C_i(|Z|+o(1))/K$. There are only finitely many types with all the
$I$ coordinates bounded by their $C_i$ and $|q_{t,Z}|\le K$.
On that finite set we may pass to the limit term by term. Taking
$j\to\infty$ and then $K\to\infty$, and using Fatou's lemma for
the remaining coordinates, gives
\begin{equation}\label{eq:plane-limit-moments}
 \sum_t\lambda_tq_{t,i}=1\quad(i\in I),\qquad
 a_i':=\sum_t\lambda_tq_{t,i}\le1\quad(i\in Z).
\end{equation}
If $Z$ is empty, the relevant type set is already finite.
The inequalities for $Z$ are essential: angles tending to zero allow
unbounded multiplicities, and their moments need not be preserved.

\smallskip\noindent\emph{Finite representation.}
Whenever $\lambda_t>0$, we have $\lambda_{j,t}>0$ for every
sufficiently large $j$. Thus this fixed type satisfies
\eqref{eq:plane-star} eventually, and
$q_t\cdot\boldsymbol\alpha^*=r_t\pi$ on taking limits.
Multiply by $\lambda_t$ and sum. All summands are nonnegative, and
\eqref{eq:plane-limit-moments} gives
\begin{equation}\label{eq:plane-r-moment}
 \pi\sum_t\lambda_t r_t
 =\sum_i\ang{i}^*\sum_t\lambda_tq_{t,i}
 =\sum_{i\in I}\ang{i}^*=(n-2)\pi.
\end{equation}
In particular $L:=\sum_t\lambda_t$ is positive and finite, since
$1\le r_t\le2$.

For the types with $\lambda_t>0$, let
$w_t=(q_t,r_t)\in\RR^{n+1}$. Equations
\eqref{eq:plane-limit-moments}--\eqref{eq:plane-r-moment} state that
\begin{equation}\label{eq:plane-target-vector}
 \sum_{\lambda_t>0}\lambda_t w_t=b,
 \qquad
 b_i=\begin{cases}1,&i\in I,\\a_i',&i\in Z,\end{cases}
 \qquad b_{n+1}=n-2.
\end{equation}
Every coordinate has finite first moment, so
$\sum_t\lambda_t\|w_t\|<\infty$.
Apply Lemma~\ref{lem:plane-barycenter} with probabilities
$\lambda_t/L$. After multiplying by $L$, we obtain finitely many
types $t_1,\ldots,t_m$ with $\lambda_{t_\ell}>0$ and weights
$\mu_\ell\ge0$ such that
\[
 \sum_{\ell=1}^m\mu_\ell w_{t_\ell}=b.
\]
Because only finitely many types were selected and each has positive
limiting frequency, their exact equations \eqref{eq:plane-star}
hold simultaneously for every sufficiently large $j$.
Weighting those equations by $\mu_\ell$ gives
\[
 \sum_{i\in I}\ang{i}_j
 +\sum_{i\in Z}a_i'\ang{i}_j=(n-2)\pi.
\]
Since $a_i'\le1$ and all current angles are positive, it follows that
\[
 \sum_{i=1}^n\ang{i}_j
 =(n-2)\pi+\sum_{i\in Z}(1-a_i')\ang{i}_j
 \ge(n-2)\pi.
\]
This contradicts \eqref{eq:plane-angle-limits} and $a_j>0$.
It proves the fixed-side gap, including all zero-angle degenerations.
\end{proof}

The finite representation is needed because $a_j$ may tend to zero
faster than the error in \eqref{eq:plane-label-balance}. It gives
exact equations for finitely many types in each sufficiently late tiling.

\subsection{A qualitative area gap on closed surfaces}
\label{sec:qualitative-closed}

The area bounds do not require orientability. To see this directly,
let $S$ be nonorientable and let
$p:\widetilde S\to S$ be its orientable double cover, endowed with
the pullback hyperbolic metric. Since each tile is an embedded closed
disc, the covering over it is trivial: its inverse image consists of
two closed discs, each mapped diffeomorphically to the original tile
with the same hyperbolic shape. Thus a tiling by $f$ congruent tiles
lifts to one by $2f$ congruent tiles, with the same tile area, genuine
side count and vertex incidences. Convexity, vertex-properness,
non-edge-to-edge incidences and
congruences that reverse orientation are all preserved. Consequently, both area
assertions of Theorem~\ref{thm:gap} reduce to the orientable case with
exactly the same constants. The surface area and tile count double,
whereas the area of each tile is unchanged.

We now give a qualitative proof of the fixed-side area gap by taking
limits of angle vectors and applying the finite corner count.
The argument uses no global orientation and applies directly to both
orientable and nonorientable surfaces.

\begin{proof}[Qualitative existence of a fixed-side area gap]
Suppose that no positive bound exists for some fixed $n$. Choose
tilings with prototiles $P_j$ whose areas $a_j>0$ tend to zero; the
surfaces may vary with $j$. Label the angles of $P_j$ by
$\ang{1}_j,\ldots,\ang{n}_j$, including equal angles separately.
Gauss--Bonnet gives
\begin{equation}\label{eq:angle-sum}
 \sum_{i=1}^n\ang{i}_j=(n-2)\pi-a_j.
\end{equation}
This formula also holds for simple nonconvex polygons
\cite[arXiv v1, Proposition 3.1]{digiosia}. Since every angle lies in
$(0,2\pi)$, a subsequence satisfies
\begin{equation}\label{eq:angle-limit}
 \ang{i}_j\longrightarrow\ang{i}^*\in[0,2\pi],
 \qquad \sum_{i=1}^n\ang{i}^*=(n-2)\pi.
\end{equation}
Only the angle vectors are taken to a limit; no compactness of the
polygons is required.

Subdivide each tiling at all tile corners, and write $\vertexset_j$ for its
vertex set. At $w\in \vertexset_j$, let $m_{i,j}(w)$ count occurrences of the
$i$th labelled angle, and let $h_j(w)$ count tiles whose side interiors
contain $w$. Then
\begin{equation}\label{eq:local-angles}
 \sum_{i=1}^n m_{i,j}(w)\ang{i}_j+h_j(w)\pi=2\pi.
\end{equation}
At least one genuine corner occurs at $w$, so $h_j(w)\in\{0,1\}$.
Thus the equation includes non-edge-to-edge incidences.

Put $I_+=\{i:\ang{i}^*>0\}$. For all sufficiently large $j$,
$\ang{i}_j\ge\ang{i}^*/2$ for every $i\in I_+$. Equation
\eqref{eq:local-angles} implies
\[
 m_{i,j}(w)\le\frac{4\pi}{\ang{i}^*}
 \qquad(i\in I_+,\ w\in \vertexset_j).
\]
Hence the patterns $((m_{i,j}(w))_{i\in I_+},h_j(w))$ belong to a
fixed finite set. For any such pattern $((m_i),h)$ with
$\sum_{i\in I_+}m_i\ang{i}^*+h\pi>2\pi$, convergence gives
$\sum_{i\in I_+}m_i\ang{i}_j+h\pi>2\pi$ for all sufficiently
large $j$. Such a pattern cannot occur in the tiling, since the omitted
angles contribute nonnegative amounts to \eqref{eq:local-angles}.
There are only finitely many patterns to exclude. Therefore, for all
sufficiently large $j$ and every $w\in \vertexset_j$,
\begin{equation}\label{eq:limit-vertex}
 \sum_{i=1}^n m_{i,j}(w)\ang{i}^*+h_j(w)\pi\le2\pi.
\end{equation}
The multiplicities of angles with zero limit need not be bounded.

Let $f_j$ be the number of tiles and $\halfcount{j}=\sum_{w\in \vertexset_j}h_j(w)$.
Each tile contributes each labelled angle exactly once, so
\begin{equation}\label{eq:corner-count}
 \sum_{w\in \vertexset_j}m_{i,j}(w)=f_j\qquad(1\le i\le n).
\end{equation}
Summing \eqref{eq:local-angles} and \eqref{eq:limit-vertex}, respectively,
and using \eqref{eq:angle-sum} and \eqref{eq:angle-limit}, we obtain
\[
 f_j\bigl((n-2)\pi-a_j\bigr)+\halfcount{j}\pi=2\pi|\vertexset_j|,
 \qquad
 f_j(n-2)\pi+\halfcount{j}\pi\le2\pi|\vertexset_j|.
\]
Subtracting gives $f_ja_j\le0$, a contradiction.
\end{proof}

The finite tiling is used in \eqref{eq:corner-count}. For a tiling of
the plane, the local angle equations alone do not give this global
identity.

\subsection{Large side counts in vertex-proper plane tilings}
\label{sec:large-side-plane}

Local finiteness was established in Section~\ref{sec:local-finite}.

Subdivide every tile side at all genuine corners of other tiles.
An open subdivided edge belongs to exactly two distinct tiles. The
following lemma supplies the planar graph needed for the growth
argument without assuming convexity.

\begin{lemma}[The dual graph]\label{lem:dual-simple}
Suppose that a vertex-proper tiling of $\HH$ consists of congruent
compact simple geodesic $n$-gons. Its subdivided tiling has a
connected, locally finite planar dual whose vertex degrees are at
least $n$ and whose faces are finite simple cycles of length at least
three. The dual has no loops, and if $n\ge6$ it has no multiple edges.
\end{lemma}
\begin{proof}
Each tile meets only finitely many other tiles, so its boundary has
a finite subdivision. Choose an interior point in each tile and one
point in each open subdivided edge. A homeomorphism from the closed
unit disc onto a tile sends radial segments to disjoint interior arcs
joining its chosen centre to these boundary points. The arcs meet
only at the centre. Joining the two arcs at each shared boundary
point gives one dual edge for each subdivided edge. This construction
uses topological arcs, which need not be geodesic.

The embedding is locally finite: every dual edge lies in its two
incident compact tiles, and a compact set meets only finitely many
tiles and finitely many arcs in each. It is connected because a path
between tile interiors can be perturbed to avoid tiling vertices and
crosses only finitely many shared edges. Its faces are the discs
around the original tiling vertices. An embedded polygon occupies
one angular sector at such a vertex, so the cyclic list of incident
tiles has no repetitions. The corresponding dual face is therefore
a finite simple cycle, with at least three edges by vertex-properness.
A tile has at least $n$ subdivided sides,
so the degree of its dual vertex, initially counting possible
multiple edges, is at least $n$. There are no loops since opposite
sides of a shared edge belong to different tiles.

Suppose that two dual edges have the same endpoints. Their union is
a Jordan curve with two boundary vertices and two boundary edges.
Its bounded closed disc contains a finite part of the dual embedding.
Let $V_\Delta,E_\Delta,F_\Delta$ count its vertices, edges and interior
faces, including the two boundary vertices and edges. This part is
connected: any component in the interior must connect to the boundary
in the full connected graph, and an edge cannot cross the Jordan
curve. The dual faces fill the disc, so
\[
 V_\Delta-E_\Delta+F_\Delta=1,
 \qquad 3F_\Delta\le2E_\Delta-2.
\]
Consequently $E_\Delta\le3V_\Delta-5$. Every interior vertex retains
all its incident edges and has degree at least $n$; the two boundary
vertices each have degree at least two in the disc. Thus
\[
 n(V_\Delta-2)+4\le2E_\Delta\le6V_\Delta-10,
 \qquad (n-6)(V_\Delta-2)\le-2.
\]
This is impossible when $n\ge6$, since $V_\Delta\ge2$. Hence the
dual is simple in that range. In particular, repeated neighbours
cannot reduce its degree below $n$.
\end{proof}

\begin{proposition}[Diameter and side count]\label{prop:plane-graph-growth}
A compact simple geodesic $n$-gon $P$ that is a vertex-proper
prototile for $\HH$, with $n\ge7$, satisfies
\begin{equation}\label{eq:plane-large-n-bound}
 D\ge\log\frac{n-3}{3},\qquad D=\diam(P).
\end{equation}
\end{proposition}
\begin{proof}
Choose a vertex-proper monohedral tiling by $P$ and use its simple
planar dual from Lemma~\ref{lem:dual-simple} as the adjacency graph.
Fix a tile $T_0$. Let $S_k$ be its vertices at
graph distance at most $k$ from $T_0$, and put $s_k=|S_k|$.
Write $e(S)$ for the number of edges with both endpoints in $S$.
Every edge incident to $S_k$ has its other endpoint in $S_{k+1}$.
Counting internal edges twice gives, for $k\ge1$,
\begin{align}
 n s_k
 &\le\sum_{w\in S_k}\deg(w)\notag\\
 &=2e(S_k)+e(S_k,S_{k+1}\setminus S_k)\notag\\
 &\le e(S_k)+e(S_{k+1})
 \le3s_k+3s_{k+1}-12.\label{eq:plane-graph-recurrence}
\end{align}
The last inequality uses $e\le3v-6$ for finite simple planar graphs;
it is applicable because $s_k\ge s_1\ge n+1\ge8$.
With $b=(n-3)/3>1$, this yields
\[
 s_{k+1}\ge b s_k+4,\qquad s_k\ge s_1 b^{k-1}.
\]

Choose $x\in T_0$. A chain of at most $k$ adjacent tiles, each of
diameter $D$, shows that every tile in $S_k$ is contained in
$B(x,(k+1)D)$. The tiles have common area $a>0$ and disjoint interiors,
so
\begin{equation}\label{eq:plane-geometric-growth}
 a s_k\le \ballarea((k+1)D),\qquad \ballarea(r)=2\pi(\cosh r-1).
\end{equation}
Take logarithms, divide by $k$, and let $k\to\infty$.
The lower bound on $s_k$ and the hyperbolic ball growth give
$\log b\le D$.
\end{proof}

Thus the diameters of vertex-proper geodesic prototiles tend to infinity
with their genuine side counts, even when reflex angles occur. The proof
uses only the planar edge bound and hyperbolic ball growth, and requires
no uniform lower bound on the angles of different prototiles.

\subsection{Large side counts on closed surfaces}\label{sec:vertices}

Define the \emph{corner degree} $d(w)$ to be the number of genuine
tile corners at $w$, excluding the flat incidence at a half vertex.
Put $h(w)=0$ at full vertices and $h(w)=1$ at half vertices.
The degree in the subdivided
tiling graph is $d(w)+h(w)$. Since each embedded tile contributes
one sector at $w$, this is also the number of distinct incident tiles.
The vertex-properness condition \eqref{eq:vertex-condition} is
therefore equivalent to $d(w)\ge3$ at full vertices and $d(w)\ge2$
at half vertices.

For edge-to-edge tilings with vertex degrees at least three, Li,
Wang, Wu and Yan \cite[arXiv v2, Lemma 2]{lisingle} prove the bound
$f\le-6\chi/(n-6)$ for $n\ge7$; see also their counting identities
\cite[arXiv v1, Section 5]{limultiple}. The following count includes
half vertices.

\begin{proposition}\label{prop:vertices}
Suppose a closed hyperbolic surface has a vertex-proper tiling by
$f$ congruent simple geodesic $n$-gons.
If $v_{\mathrm{half}}$ is the number of half vertices, then the tile area satisfies
\begin{equation}\label{eq:euler-area}
 \Area(P)\ge\frac\pi3\left(n-6+\frac{v_{\mathrm{half}}}{f}\right)
 \ge\frac\pi3(n-6).
\end{equation}
\end{proposition}

\begin{proof}
Let $v_{\mathrm{full}}$ and $v_{\mathrm{half}}$ be the numbers of full and half vertices, and let
$v=v_{\mathrm{full}}+v_{\mathrm{half}}$ and $e$ be the vertex and edge counts after subdivision at all
corners. Each tile contributes $n$ genuine corners. Vertex-properness
therefore gives
\begin{equation}\label{eq:full-half-corners}
 nf=\sum_{w\ \mathrm{full}}d(w)+\sum_{w\ \mathrm{half}}d(w)
 \ge3v_{\mathrm{full}}+2v_{\mathrm{half}}.
\end{equation}
Each half vertex adds one subdivision point on the boundary of the
tile passing through it. Thus the total number of subdivided tile
sides is $nf+v_{\mathrm{half}}$. Every graph edge belongs to two tile boundaries, so
\begin{equation}\label{eq:full-half-edges}
 2e=nf+v_{\mathrm{half}}.
\end{equation}
Equivalently, full vertices have graph degree $d(w)\ge3$, and half
vertices have graph degree $d(w)+1\ge3$. Combining
\eqref{eq:full-half-corners} and \eqref{eq:full-half-edges} gives
\[
 3v=3v_{\mathrm{full}}+3v_{\mathrm{half}}\le nf+v_{\mathrm{half}}=2e.
\]
Euler's formula yields
\begin{align}
 \chi=v-e+f
 &\le f-\frac e3
  =f-\frac{nf+v_{\mathrm{half}}}{6}\notag\\
 &=\frac{(6-n)f-v_{\mathrm{half}}}{6}.
 \label{eq:full-half-euler}
\end{align}
Since $f\Area(P)=-2\pi\chi$, multiplication by $-2\pi/f$ gives
\eqref{eq:euler-area}.
\end{proof}

Equality in the first inequality of \eqref{eq:euler-area} holds
exactly when every full vertex has corner degree three and every half
vertex has corner degree two. Equality in the bound
$\Area(P)\ge\pi(n-6)/3$ also requires $v_{\mathrm{half}}=0$. For $n\ge7$,
\eqref{eq:full-half-euler} gives
\[
 (n-6)f+v_{\mathrm{half}}\le-6\chi,\qquad
 f\le\frac{-6\chi-v_{\mathrm{half}}}{n-6}\le\frac{-6\chi}{n-6}.
\]
When $v_{\mathrm{half}}=0$, this recovers the upper bound and equality condition in
\cite[arXiv v2, Lemma 2]{lisingle}.

For convex tiles, the Euler area estimate also appears in
\cite[arXiv v1, Proposition 4.4]{hirsch}.

\subsection{Uniform gaps for vertex-proper tilings}
\label{sec:qualitative-uniform}

We now combine the fixed-side qualitative gaps with the preceding
large-side estimates. By Proposition~\ref{prop:plane-graph-growth}, a
sequence of vertex-proper plane prototiles with diameters tending to
zero has bounded side counts. A subsequence has one fixed side count,
contradicting Proposition~\ref{prop:plane-fixed-n-qualitative}. This
proves the uniform plane diameter gap.

On a closed surface, Proposition~\ref{prop:vertices} gives
$a\ge\pi(n-6)/3$. Hence a sequence of vertex-proper tilings with
tile areas tending to zero eventually has $n\le6$. A fixed-side
subsequence contradicts the qualitative area gap from
Section~\ref{sec:qualitative-closed}. This proves the uniform
closed-surface area gap and completes the qualitative proofs of
both main theorems.

\subsection{The integer gap and well-quasi-orders}
\label{sec:qualitative-integer}

The two quantitative arguments use the following integer-gap principle.
We first establish its existence; Section~\ref{sec:integer} gives
explicit constants.

\begin{lemma}[Qualitative integer gap]\label{lem:qualitative-gap}
For each positive integer $k$ there is $\eta_k>0$ with the following
property. Let $T$ be a finite matrix with nonnegative integer entries,
let every entry of $b$ be $1$ or $2$, and suppose
$Q=\{y\in\RR_{\ge0}^k:Ty\le b\}$ is bounded.
If $\mu=\max_{y\in Q}\sum_i y_i<\ell$ for a positive integer $\ell$,
then $\ell-\mu\ge\eta_k$.
\end{lemma}
\begin{proof}
First normalise every inequality $a\cdot x\le b$ by writing it as
$(2a/b)\cdot x\le2$. Because $b$ is either $1$ or $2$, the new
coefficient vector still has nonnegative integer entries. Let $S$
be the finite set of these coefficient vectors, and write
\[
 Q(S)=\{x\in\RR_{\ge0}^k:s\cdot x\le2\text{ for every }s\in S\}.
\]
Define $S\preceq T$ to mean that every vector in $S$ is bounded
coordinate by coordinate by some vector in $T$. A consequence of
Higman's lemma \cite{higman} is that, in every infinite sequence
$S_1,S_2,\ldots$, there are indices $p<q$ with $S_p\preceq S_q$;
see \cite[Section 2]{abriola}. This is the well-quasi-order property
needed here.

To see its effect on our inequalities, suppose $S\preceq T$ and
$x\in Q(T)$. Given $s\in S$, choose $t\in T$ with $s_i\le t_i$
for every $i$. Since $x_i\ge0$, we have
$s\cdot x\le t\cdot x\le2$. Thus $x\in Q(S)$, and consequently
\[
 Q(T)\subseteq Q(S),\qquad \mu(T)\le\mu(S),
\]
where $\mu(S)$ is the maximum coordinate sum on the bounded set
$Q(S)$.

If these maxima had an infinite strictly increasing sequence, the
well-quasi-order property would give $p<q$ with $S_p\preceq S_q$,
forcing $\mu(S_q)\le\mu(S_p)$, a contradiction. In particular, the
maxima cannot approach a fixed integer from below: such an approach
would contain a strictly increasing subsequence. Finally, boundedness
forces each coordinate to occur with a positive integer coefficient
in some inequality, so each coordinate is at most $2$ and
$\mu\le2k$. Only the integers $1,\ldots,2k$ need be considered;
larger integers are separated from $\mu$ by at least $1$.
This proves the existence of a gap depending only on $k$.
\end{proof}

\section{An explicit integer gap}
\label{sec:integer}

We now make Lemma~\ref{lem:qualitative-gap} quantitative.
The surface and plane arguments lead to finite systems of inequalities with
nonnegative integer coefficients and right-hand sides $1$ or $2$.
The integers count occurrences of labelled corners; the angle values
themselves need not be rational. We bound the distance between the
maximum coordinate sum and a larger positive integer.

Define
\begin{equation}\label{eq:denominators}
 B_0=1,\qquad B_{j+1}=B_j(2B_j+1),\qquad
 \varepsilon_n=\frac1{B_{\lceil n/2\rceil+1}}.
\end{equation}
The first terms of $B_j$ are $1,3,21,903,1631721$.

\subsection{A gap estimate for integer inequalities}

The bound depends on the number of variables and on the integer target.

\begin{lemma}[Integer gap estimate]\label{lem:gap}
Let $k\ge1$ and let $Q=\{y\in\RR_{\ge0}^k:Ty\le b\}$ be bounded.
Assume that $T$ is a finite matrix of nonnegative integers and that
each entry of $b$ is $1$ or $2$. Put $\mu=\max_Q\sum_i y_i$.
For every positive integer $\ell$ with $\mu<\ell$,
\begin{equation}\label{eq:sharp-gap}
 \ell-\mu\ge G(k,\ell):=
 \begin{cases}
 B_{k-\lfloor\ell/2\rfloor}^{-1},&1\le\ell\le2k,\\
 \ell-2k,&\ell>2k.
 \end{cases}
\end{equation}
For each pair $(k,\ell)$, equality occurs for a diagonal system.
\end{lemma}

\begin{proof}
We first prove the inequality and then construct equality examples.

\medskip\noindent
\emph{(I) Normalisation.}
Double every row whose right side is $1$. All right sides are now $2$,
and all coefficients are still nonnegative integers. Boundedness
forces each coordinate to occur in some row with coefficient at least
$1$. Hence $Q\subseteq[0,2]^k$.

If $\ell>2k$, then $\ell-\mu\ge\ell-2k$ immediately. If $\ell=2k$,
suppose that $\mu>2k-1$. Every coordinate of a maximiser must exceed
$1$: otherwise its total would be at most $1+2(k-1)=2k-1$.
A nonzero constraint row must therefore have just one positive entry,
equal to $1$. Two positive entries, or one entry at least $2$, would
make its left side exceed $2$. Boundedness then forces precisely the
coordinate bounds $y_i\le2$, so $Q=[0,2]^k$ and $\mu=2k$.
This contradicts $\mu<\ell$. Thus the gap for $\ell=2k$ is at least
$1=B_0^{-1}$.

\medskip\noindent
\emph{(II) Coordinate faces.}
Write $E_j=1/B_j$. The recurrence in \eqref{eq:denominators} gives
\begin{equation}\label{eq:gap-recurrence}
 E_0=1,\qquad E_{j+1}=\frac{E_j^2}{2+E_j},\qquad
 E_j-E_{j+1}=\frac{2E_{j+1}}{E_j}.
\end{equation}
For a fixed $\ell$, we use induction on $k$. In dimension zero the
maximum is $0$, so set $G(0,\ell)=\ell$. Part (I) treats
$\ell\ge2k$. Suppose $\ell<2k$, and put
\[
 s=\lfloor\ell/2\rfloor,\qquad r=k-s\ge1,\qquad
 g_j=G(j,\ell)\quad(0\le j<k).
\]
Then $g_{k-1}=E_{r-1}=:\eta$. Set $\delta=\ell-\mu>0$.
If $\delta\ge \eta$, the conclusion follows from $\eta>E_r$.
We may therefore assume $\delta<\eta$.

Choose a maximising vertex of $Q$ and relabel its coordinates so that
$y_1\le\cdots\le y_k$. Such a vertex exists because $Q$ is a nonempty
bounded polytope. For $1\le j<k$, set its first $j$ coordinates to
zero. Nonnegative coefficients ensure that the resulting point is
still feasible. Deleting these zero coordinates gives a bounded
problem in $k-j$ variables. Extending a feasible point by zero in
the deleted coordinates gives a point of $Q$, so this smaller problem
has maximum at most $\mu<\ell$. Its maximum is therefore at most
$\ell-g_{k-j}$ by induction. Therefore
\begin{equation}\label{eq:prefix-bounds}
 \sum_{i=1}^j y_i\ge g_{k-j}-\delta\quad(1\le j<k),
 \qquad \sum_{i=1}^k y_i=\ell-\delta.
\end{equation}
In particular $y_1\ge \eta-\delta>0$ when $k>1$.
For $k=1$, we have $\ell=1$ and $y_1=1-\delta>0$.
All coordinates of the maximising vertex are thus positive.

\medskip\noindent
\emph{(III) Product comparison.}
Define a positive vector
\[
 z=(g_{k-1}-\delta,\ g_{k-2}-g_{k-1},\ldots,g_0-g_1).
\]
For $k=1$ this means $z=(\ell-\delta)$. The quantities $g_j$ strictly
decrease with $j$, so all the displayed differences are positive.
The first $j$ entries of $z$ sum to $g_{k-j}-\delta$, and its total
is $\ell-\delta$. Thus, if
$S_j=\sum_{i=1}^j(y_i-z_i)$, then \eqref{eq:prefix-bounds} says
$S_j\ge0$ for $j<k$ and $S_k=0$.

We claim that
\begin{equation}\label{eq:product-comparison}
 \prod_{i=1}^k y_i\ge\prod_{i=1}^k z_i.
\end{equation}
The tangent-line inequality for the concave function $\log$ gives
$\log y_i-\log z_i\ge(y_i-z_i)/y_i$. Summing and grouping successive
prefix sums yields
\[
 \sum_i(\log y_i-\log z_i)
 \ge\sum_i\frac{y_i-z_i}{y_i}
 =\sum_{j=1}^{k-1}
   \left(\frac1{y_j}-\frac1{y_{j+1}}\right)S_j\ge0.
\]
The last inequality uses the ordering of $y_i$. Exponentiation proves
\eqref{eq:product-comparison}; no ordering of $z_i$ is needed.

\medskip\noindent
\emph{(IV) Evaluation of the product.}
Let $p=\gcd(\ell,2)$, so $p=1$ for odd $\ell$ and $p=2$ for even
$\ell$. For $j<s$, the differences $g_j-g_{j+1}$ are all $2$ when
$\ell$ is odd. When $\ell$ is even, they are $2$ except for the last
one, which is $1$. Their product is therefore $2^s/p$.

The other differences are $E_j-E_{j+1}$ for $0\le j\le r-2$.
By \eqref{eq:gap-recurrence}, their product telescopes:
\[
 \prod_{j=0}^{r-2}(E_j-E_{j+1})
 =\prod_{j=0}^{r-2}\frac{2E_{j+1}}{E_j}
 =2^{r-1}E_{r-1}=2^{r-1}\eta.
\]
Empty products equal $1$, which includes $r=1$ and $s=0$.
Multiplying by the first entry $\eta-\delta$ of $z$ gives
\begin{equation}\label{eq:product-lower}
 \prod_i y_i\ge\prod_i z_i
 =\frac{2^{k-1}}p\,\eta(\eta-\delta).
\end{equation}

\medskip\noindent
\emph{(V) The determinant bound.}
At a vertex there are $k$ independent active constraints. None is a
coordinate constraint $y_i=0$, since every $y_i$ is positive. Choose
$k$ active rows of the normalised matrix and write
\[
 Cy=2\mathbf1,\qquad D=|\det C|\in\ZZ_{>0}.
\]
Cramer's rule gives $Dy_i\in2\ZZ$. Consequently
$D\delta=\ell D-\sum_i Dy_i$ is a positive integer divisible by $p$.
It follows that $\delta\ge p/D$.

Multiply column $i$ of $C$ by $y_i$ to obtain a matrix $L$.
Its entries are nonnegative and every row sums to $2$.
The determinant expansion is bounded by the product of row sums:
\[
 D\prod_i y_i=|\det L|
 \le\sum_{\sigma\in\mathfrak S_k}\prod_{i=1}^k L_{i,\sigma(i)}
 \le\prod_{i=1}^k\left(\sum_j L_{ij}\right)=2^k.
\]
Indeed, the product of row sums contains every permutation term in
the middle expression and only adds nonnegative terms. Hence
\begin{equation}\label{eq:det-gap}
 \delta\ge\frac pD\ge\frac p{2^k}\prod_i y_i
 \ge\frac{\eta(\eta-\delta)}2.
\end{equation}
Rearranging the last inequality gives
\[
 (2+\eta)\delta\ge \eta^2,\qquad
 \delta\ge\frac{\eta^2}{2+\eta}=E_r.
\]
This proves the required gap and completes the induction.

\medskip\noindent
\emph{(VI) Equality examples.}
Set $a_j=2B_{j-1}+1$ for $j\ge1$. The diagonal constraints
$a_j u_j\le2$ in $r$ coordinates have maximum sum
\[
 \sum_{j=1}^r\frac2{a_j}
 =\sum_{j=1}^r(E_{j-1}-E_j)=1-E_r.
\]
If $\ell=2s+1\le2k$, add $s$ coordinates bounded by $2$.
If $\ell=2s\le2k$, add $s-1$ coordinates bounded by $2$ and one
bounded by $1$ (the constraint $2u\le2$). With $r=k-s$, both systems
have $k$ coordinates and maximum $\ell-E_r$.
The even case also allows $r=0$. Finally, for $\ell>2k$, the cube
$[0,2]^k$ attains the gap $\ell-2k$.
\end{proof}

For diagonal constraints $m_i y_i\le2$, the optimum is
$\mu=2\sum_{i=1}^k m_i^{-1}$. The sharp diagonal case of
Lemma~\ref{lem:gap} is already contained in
Ambro and Barc\u au~\cite[Theorem~0.2(a)]{ambrobarcau}.
Diagonal constraints with right-hand side $1$ are included by
doubling the corresponding coefficient.
To compare the constants for $1\le\ell\le2k$, write
$r=k-\lfloor\ell/2\rfloor$ and
set $\delta=k-\ell/2$ in their theorem. Their sequence satisfies
\[
 u_{1,q}=q,\qquad u_{j+1,q}=u_{j,q}(u_{j,q}+1),\qquad
 u_{r+1,2}=u_{r+2,1}=2B_r.
\]
If $\ell$ is odd, take $q=2$; if $\ell$ is even, take $q=1$.
For $1\le\ell\le2k$, their bound then gives
\[
 \ell-\mu\ge
 \begin{cases}
  2/u_{r+1,2},&\ell\text{ odd},\\
  2/u_{r+2,1},&\ell\text{ even},
 \end{cases}
 \ =\frac1{B_r}.
\]
For $\ell>2k$, the bound follows from $\mu\le2k$.
The preceding proof extends the same sharp constants to arbitrary
coupled nonnegative integer constraints.

The equality examples involve the greedy denominators
$3,7,43,1807,\ldots$ for underapproximating $1/2$;
see also Nathanson~\cite[Theorem~5]{nathanson}.
The product comparison in the proof is related to the majorization
arguments of Soundararajan~\cite{soundararajan} and
Ambro and Barc\u au~\cite[Lemmas~1.1--1.2]{ambrobarcau}.
Related product--sum inequalities and Sylvester-sequence bounds for
barycentric coordinates occur in the work of Averkov, Kr\"umpelmann
and Nill~\cite[Section~4]{averkovkrumpelmannnill} on integral simplices.
Our proof uses the active-constraint determinant to obtain the
estimate for coupled constraints directly. Sharpness here concerns the
arithmetic problem; geometric realisability of its equality systems
is a separate question.

\section{Explicit diameter bounds in the plane}
\label{sec:convex-diameter}

Proposition~\ref{prop:plane-graph-growth} bounds the diameter of a
vertex-proper plane prototile in terms of its side count. We now
derive an explicit bound for every fixed number of sides, without
assuming vertex-properness. This treats nonconvex polygons and
requires no closed quotient or symmetry. Combining the two estimates
will give an explicit uniform bound for vertex-proper prototiles.

\subsection{The fixed-side covering problem}

Let $P$ be a compact simple geodesic $n$-gonal monotile. Put
$D=\diam(P)>0$, $a=\Area(P)>0$, $x_i=\ang{i}/\pi$, and $\ell=n-2$.
Allowing reflex angles, Gauss--Bonnet gives
\begin{equation}\label{eq:quant-angle-sum}
 \sum_i x_i=\ell-\anorm<\ell.
\end{equation}
Choose an isometry from $P$ onto each copy to transport the labels.
Each copy then has each label exactly once, even when angles are equal;
the choices need not be consistent between copies.

At a genuine corner centre, take the complete star of incident tiles.
Near that point, each genuine corner contributes one angular sector
of size $\ang{i}\in(0,2\pi)$, and each side-interior incidence a flat
sector of size $\pi$. Local finiteness makes the sectors a finite
partition of the tangent circle. Two flat sectors would leave no room
for the positive genuine corner, so there is at most one flat incidence.
Each actual star type therefore satisfies
\begin{equation}\label{eq:quant-stars}
 q_t\cdot x=r_t,\qquad
 q_t\in\ZZ_{\ge0}^n\setminus\{0\},\quad r_t\in\{1,2\}.
\end{equation}
There are finitely many types for each fixed tile, since
$q_{t,i}\le2/x_i$ and all $x_i>0$. This requires no uniform lower
angle bound over different tiles.

Consider the covering program and its dual:
\begin{align}
 \mu&=\min\left\{\sum_i y_i:y\ge0,
                 \ q_t\cdot y\ge r_t\text{ for every }t\right\},
                   \label{eq:quant-cover}\\
 \mu&=\max\left\{\sum_t\lambda_t r_t:\lambda_t\ge0,
                 \ \sum_t\lambda_tq_{t,i}\le1\text{ for every }i\right\}.
                   \label{eq:quant-dual}
\end{align}
The angle vector $x$ is feasible, so
$\mu\le\sum_i x_i<\ell$. An optimum exists
after restricting to the compact set $\sum_i y_i\le\sum_i x_i$;
finite-dimensional linear-programming duality supplies a dual optimum.

\begin{lemma}[A packing polytope for the covering optimum]
\label{lem:cover-packing}
For the finite system \eqref{eq:quant-cover}, there is a bounded packing
polytope with at most $n$ variables, nonnegative integer coefficients,
right-hand sides $1$ or $2$, and maximum coordinate sum equal to $\mu$.
Consequently $\mu\le\ell-\varepsilon_n$.
\end{lemma}
\begin{proof}
Take optimal $y^*$ and $\lambda^*$, and let $I=\{i:y_i^*>0\}$.
It is nonempty because the covering right-hand sides are positive.
Complementary slackness gives
\[
 \sum_t\lambda_t^*q_{t,i}=1\quad(i\in I),\qquad
 q_t\cdot y^*=r_t\quad(\lambda_t^*>0).
\]
Use only the positive-weight rows, restricted to these coordinates:
\[
 Q_I=\{z\in\RR_{\ge0}^{I}:q_{t,I}\cdot z\le r_t
                  \text{ whenever }\lambda_t^*>0\}.
\]
For each $i\in I$, the identity
$\sum_t\lambda_t^*q_{t,i}=1$ supplies a row with
$\lambda_t^*>0$ and $q_{t,i}\ge1$. Every $z\in Q_I$ therefore
satisfies $0\le z_i\le r_t/q_{t,i}\le2$, so $Q_I$ is bounded.
The vector $y_I^*$ is feasible. For any $z\in Q_I$,
\[
 \sum_{i\in I}z_i
 =\sum_t\lambda_t^*(q_{t,I}\cdot z)
 \le\sum_t\lambda_t^*r_t=\mu=\sum_{i\in I}y_i^*.
\]
Its packing maximum is therefore exactly $\mu$. Apply
Lemma~\ref{lem:gap} with $k=|I|\le n$ and $\ell=n-2$.
If $\ell\le2k$, monotonicity of $B_j$ gives
$G(k,\ell)\ge B_{n-\lfloor\ell/2\rfloor}^{-1}=\varepsilon_n$;
if $\ell>2k$, its gap is at least $1\ge \varepsilon_n$. Thus
\begin{equation}\label{eq:quant-integer-bound}
 \mu\le\ell-\varepsilon_n.
\end{equation}
\end{proof}

The bound \eqref{eq:quant-integer-bound} concerns the auxiliary optimum
$\mu$. From $\mu\le\sum_i x_i$ one cannot
deduce $\sum_i x_i\le\ell-\varepsilon_n$, or $a\ge\pi \varepsilon_n$.
We compare $\mu$ with the actual angle sum by counting stars in balls.

Reversing the covering inequalities need not preserve the optimum.
For the single constraint $3y_1+y_2+y_3\ge2$, the minimum coordinate
sum is $2/3$, attained at $(2/3,0,0)$. Reversing the inequality in all
three coordinates gives $3z_1+z_2+z_3\le2$, whose maximum sum is $2$,
not $2/3$. The lemma retains $I=\{1\}$ and the positive
dual weight $\lambda=1/3$, obtaining $3z_1\le2$ with maximum $2/3$.
This is an algebraic example; no geometric realisation is assumed.

\subsection{Star counts and the diameter bound}

Let $\ballarea(R)=2\pi(\cosh R-1)$ and take $R>D$. Let $c_i(R)$ count all
occurrences of labelled corner $i$ in the closed ball $B(R)$.
Every tile having this corner in $B(R)$ lies in $B(R+D)$; every tile
meeting $B(R-D)$ has this corner in $B(R)$. Disjoint interiors imply
\begin{equation}\label{eq:quant-ball-count}
 \frac{\ballarea(R-D)}a\le c_i(R)\le\frac{\ballarea(R+D)}a.
\end{equation}
Only the diameter is used in these inclusions.

Let $\starcount_t(R)$ count complete stars of type $t$ whose centres lie in
$B(R)$. Stars centred near the boundary are still counted in full.
Then $c_i(R)=\sum_t\starcount_t(R)q_{t,i}$. The weights
$\eta_t(R)=a\starcount_t(R)/\ballarea(R+D)$ are dual feasible and satisfy
\[
 \frac{\ballarea(R-D)}{\ballarea(R+D)}
 \le\sum_t\eta_t(R)q_{t,i}\le1.
\]
Using \eqref{eq:quant-stars} gives
\[
 \mu\ge\sum_t\eta_t(R)r_t
 =\sum_i x_i\sum_t\eta_t(R)q_{t,i}
 \ge\frac{\ballarea(R-D)}{\ballarea(R+D)}\sum_i x_i.
\]
As $R\to\infty$, the ratio tends to $e^{-2D}$. Thus
\begin{equation}\label{eq:quant-boundary-loss}
 \boxed{\mu\ge e^{-2D}(\ell-\anorm).}
\end{equation}
In particular, enlarging the ball does not remove the boundary loss.

\begin{proposition}[Diameter bound for a fixed number of sides]
\label{prop:plane-fixed-n}
Every compact simple geodesic $n$-gonal monotile of $\HH$, $n\ge3$,
satisfies
\begin{equation}\label{eq:quant-plane-diameter}
 D>\dplane:=-\log\left(1-\frac{\varepsilon_n}{2(n-2)}\right).
\end{equation}
The tiles may be nonconvex. No symmetry, edge-to-edge, uniform angle,
or aspect-ratio assumption is required.
\end{proposition}
\begin{proof}
Combining \eqref{eq:quant-integer-bound} and
\eqref{eq:quant-boundary-loss}, and using $P\subset B(p,D)$ for $p\in P$,
gives two necessary area estimates:
\begin{equation}\label{eq:quant-area-sandwich}
 \underbrace{\ell-(\ell-\varepsilon_n)e^{2D}}_{L_n(D)}
 \le\anorm\le
 \underbrace{2(\cosh D-1)}_{U(D)}.
\end{equation}
The left estimate gives a positive area bound only when $L_n(D)>0$;
in all cases $a>0$. For fixed $n$, as $D\to0$,
$L_n(D)=\varepsilon_n-2(\ell-\varepsilon_n)D+O(D^2)$ and
$U(D)=D^2+O(D^4)$.
Set $t=e^{-D}\in(0,1)$. Then
\begin{align*}
 \varepsilon_n&\le\ell-\mu
 \le\ell(1-t^2)+2t^2(\cosh D-1)\\
 &=\ell(1-t^2)+t(1-t)^2\\
 &=(1-t)[\ell(1+t)+t(1-t)]<2\ell(1-t).
\end{align*}
The last step is strict because
$2\ell-[\ell(1+t)+t(1-t)]=(1-t)(\ell-t)>0$.
Therefore $e^{-D}<1-\varepsilon_n/(2\ell)$, which proves
\eqref{eq:quant-plane-diameter}.
\end{proof}

\subsection{A uniform bound for vertex-proper prototiles}

\begin{proof}[Completion of the proof of Theorem~\ref{thm:convex-diameter}]
For $3\le n\le6$ use Proposition~\ref{prop:plane-fixed-n}; for $n\ge7$
use Proposition~\ref{prop:plane-graph-growth}. Thus an explicit choice is
\begin{equation}\label{eq:plane-universal-constant}
 d_*:=\min\{d_3^{\mathrm{pl}},d_4^{\mathrm{pl}},d_5^{\mathrm{pl}},
             d_6^{\mathrm{pl}},\log(4/3)\}
 =\log\frac{13053768}{13053767}>0.
\end{equation}
For $3\le n\le6$, Proposition~\ref{prop:plane-fixed-n} gives
$D>d_n^{\mathrm{pl}}\ge d_*$. For $n\ge7$,
Proposition~\ref{prop:plane-graph-growth} gives
$D\ge\log(4/3)>d_*$. Thus every vertex-proper plane prototile
satisfies $D>d_*$, as asserted.
\end{proof}

The proof uses the arithmetic lemma from Section~\ref{sec:integer}
and the plane star count. It requires no finite quotient.

\begin{proposition}\label{prop:convex-body-monotile}
Every compact convex monotile of $\HH$ with nonempty interior is a
finite geodesic polygon. Consequently Theorem~\ref{thm:convex-diameter}
also applies without initially assuming a polygonal boundary.
\end{proposition}
\begin{proof}
Convex bodies have boundaries of area zero, so the area argument in
Section~\ref{sec:local-finite} gives local finiteness.
Only finitely many other tiles meet a fixed tile $K$, and each
intersection is a segment, a point, or empty by convexity.
Approaching any point of $\partial K$ from outside $K$, local
finiteness and closedness show that the point belongs to another tile.
Thus $\partial K$ is covered by finitely many geodesic segments and
points. Any point outside the union of those segments would be
isolated in $\partial K$, which is impossible for the boundary of a
convex body with nonempty interior. Merging collinear segments gives
a finite geodesic polygon, to which the theorem applies.
\end{proof}

The convex tiles of arbitrarily small area constructed by Margulis
and Mozes \cite{margulis}, as recalled in
\cite[arXiv v1, discussion following Corollary 4.10]{digiosia}, therefore cannot have
diameters tending to zero, even if their side counts vary.
Section~\ref{sec:classical-diameters} computes the diameters of
classical convex examples of small area. The general bound above
does not determine the limit for any particular family.

\section{Explicit area bounds on closed surfaces}
\label{sec:gap}\label{sec:explicit}

In this section $S$ is a closed hyperbolic surface with
$\chi=\chi(S)<0$, tiled by finitely many congruent simple geodesic
polygons. The surface may be orientable or nonorientable.
Section~\ref{sec:qualitative-closed} proved the fixed-side area gap
qualitatively. We now use the exact label identity \eqref{eq:label-count}
to obtain explicit bounds, still without a convexity assumption.

\begin{theorem}\label{thm:explicit}
For every $n\ge3$, the fixed-side constant in Theorem~\ref{thm:gap}
may be chosen as $c_n=\pi \varepsilon_n$. In particular,
\begin{equation}\label{eq:four-constants}
 \boxed{c_3=c_4=\frac\pi{903},\qquad
 c_5=c_6=\frac\pi{1631721}.}
\end{equation}
\end{theorem}

\subsection{Angle equations and integer inequalities}\label{sec:angle-system}

Let a closed hyperbolic surface be tiled by $f$ congruent copies of
a simple geodesic $n$-gon $P$. Write $a=\Area(P)>0$, label its angles
$\ang{1},\ldots,\ang{n}$, and put $x_i=\ang{i}/\pi$. Then
\begin{equation}\label{eq:GB}
 x_1+\cdots+x_n=n-2-\frac a\pi.
\end{equation}
As before, write $h(w)=0$ at a full vertex and $h(w)=1$ at a half vertex.
At a tiling vertex $w$, let $m_i(w)$ count occurrences of the labelled
angle $\ang{i}$. The local angle equation, divided by $\pi$, is
\begin{equation}\label{eq:balance}
 \sum_{i=1}^n m_i(w)x_i=2-h(w)\in\{1,2\}.
\end{equation}
The coefficients $m_i(w)$ are nonnegative integers. Each of the $f$
tiles contributes its $i$th labelled angle exactly once, so
\begin{equation}\label{eq:label-count}
 \sum_w m_i(w)=f\qquad(1\le i\le n).
\end{equation}

Consider the polytope
\begin{equation}\label{eq:angle-polytope}
 Q=\left\{y\in\RR_{\ge0}^n:
       \sum_{i=1}^n m_i(w)y_i\le2-h(w)\text{ for every }w\right\}.
\end{equation}
The actual angle vector $x=(x_1,\ldots,x_n)$ belongs to $Q$ and
satisfies every vertex inequality with equality. The polytope is bounded:
for each $i$, \eqref{eq:label-count} supplies a vertex $w$ with
$m_i(w)\ge1$. Since all summands are nonnegative, every $y\in Q$
satisfies
\[
 0\le y_i\le\frac{2-h(w)}{m_i(w)}\le2.
\]
Thus $Q\subseteq[0,2]^n$.

Moreover, $x$ maximises the coordinate sum on $Q$. For any $y\in Q$,
summing over all tiling vertices gives
\begin{align}
 f\sum_{i=1}^n y_i
 &=\sum_w\sum_{i=1}^n m_i(w)y_i\notag\\
 &\le\sum_w(2-h(w))
   =\sum_w\sum_{i=1}^n m_i(w)x_i
   =f\sum_{i=1}^n x_i.\label{eq:angle-maximum}
\end{align}
Here the equalities follow from \eqref{eq:label-count} and
\eqref{eq:balance}. Since $f>0$, the maximum coordinate sum $\mu$ on
$Q$ satisfies
\begin{equation}\label{eq:area-as-gap}
 \mu=n-2-\frac a\pi<n-2,
 \qquad \frac a\pi=(n-2)-\mu.
\end{equation}

\subsection{Explicit area and diameter bounds}

\begin{proof}[Proof of Theorem~\ref{thm:explicit}]
The polytope in \eqref{eq:angle-polytope} has $k=n$ variables and
satisfies the hypotheses of Lemma~\ref{lem:gap}.
By \eqref{eq:area-as-gap}, its maximum is
$\mu=n-2-a/\pi<n-2$. Apply the lemma with $\ell=n-2$:
\[
 \frac a\pi\ge\frac1{B_{n-\lfloor(n-2)/2\rfloor}}
 =\frac1{B_{\lceil n/2\rceil+1}}=\varepsilon_n.
\]
Thus $a\ge\pi\varepsilon_n$. The four values in
\eqref{eq:four-constants} follow from $B_3=903$ and $B_4=1631721$.
\end{proof}

Let $D_i^S=\diam_S(T_i)$ be the ambient diameter of a tile, and choose
$p\in T_i$. The tile is contained in $B_S(p,D_i^S)$. A hyperbolic ball
of radius $D_i^S$ in the universal cover maps onto this surface ball,
so
\[
 a\le\Area B_S(p,D_i^S)\le2\pi(\cosh D_i^S-1).
\]
Combining this with $a\ge\pi \varepsilon_n$ gives
\begin{equation}\label{eq:diameter-explicit}
 D_i^S\ge\dclosed:=\arcosh\left(1+\frac{\varepsilon_n}{2}\right).
\end{equation}
The estimate is independent of the injectivity radius and also holds
for the diameter of the lifted prototile.

\subsection{Uniform bounds and tile counts}

\begin{corollary}\label{cor:unrestricted}
Every tile in a finite vertex-proper monohedral geodesic tiling of
a closed hyperbolic surface satisfies
\begin{equation}\label{eq:valence-explicit}
 \Area(P)\ge\frac\pi{1631721},\qquad
 D_i^S\ge\arcosh\left(1+\frac1{3263442}\right).
\end{equation}
The bounds are independent of the number of sides, and the number
of tiles satisfies $f\le-3263442\chi(S)$. In particular, these
conclusions hold for convex geodesic tiles. Reflex angles are allowed
in a vertex-proper tiling.
\end{corollary}
\begin{proof}
For $n\ge7$, Proposition~\ref{prop:vertices} gives $a\ge\pi/3$.
For $3\le n\le6$, Theorem~\ref{thm:explicit} gives
$a\ge\pi/1631721$. The ball estimate proves the diameter bound, and
$fa=-2\pi\chi(S)$ gives the tile-count bound.
\end{proof}

This completes the proof of Theorem~\ref{thm:gap}.

By \eqref{eq:surface-area}, the common tile area satisfies
$fa=-2\pi\chi$. The fixed-side theorem therefore gives
\begin{equation}\label{eq:tile-number}
 f\le-2\chi B_{\lceil n/2\rceil+1}.
\end{equation}
For triangles and quadrilaterals this is $f\le-1806\chi$.
For vertex-proper tilings, the Euler estimate for $n\ge7$ is
\[
 (n-6)f+v_{\mathrm{half}}\le-6\chi.
\]
For an orientable surface of genus $g\ge2$, these become
\[
 f\le4(g-1)B_{\lceil n/2\rceil+1},\qquad
 (n-6)f+v_{\mathrm{half}}\le12(g-1),
\]
with $f\le3612(g-1)$ for triangles and quadrilaterals.
For a nonorientable surface of genus $h\ge3$, the corresponding
bounds are $f\le2(h-2)B_{\lceil n/2\rceil+1}$ and
$(n-6)f+v_{\mathrm{half}}\le6(h-2)$.
In the orientable case with $v_{\mathrm{half}}=0$, the Euler bound is
the count in \cite[arXiv v2, Lemma 2]{lisingle}.
A one-face polygonal presentation in that classification need not
be an embedded closed-disc tile in our convention. The counting
identity is compatible with both conventions when incidences are
counted with multiplicity.

\subsection{Comparison of the explicit bounds}\label{sec:bound-comparison}

For $\ell=n-2$ and $\varepsilon_n=B_{\lceil n/2\rceil+1}^{-1}$,
the estimates proved above are
\begin{equation}\label{eq:intro-two-bounds}
\begin{aligned}
\text{closed surface:}\qquad
 a&\ge\pi \varepsilon_n,&
 \diam_S(T_i)&\ge\dclosed:=\arcosh(1+\varepsilon_n/2),\\
\text{entire plane:}\qquad
 D&>\dplane:= -\log\left(1-\frac{\varepsilon_n}{2\ell}\right).
\end{aligned}
\end{equation}

On a closed surface, the actual angle vector maximises the packing
objective, giving $a/\pi=\ell-\mu$ exactly. In the plane, the covering
optimum need not equal the actual angle sum; the available comparison
is \eqref{eq:quant-boundary-loss}. As $\varepsilon_n\to0$, the two
diameter estimates have orders $\sqrt{\varepsilon_n}$ and
$\varepsilon_n/[2(n-2)]$, respectively.

\begin{center}
\small\renewcommand{\arraystretch}{1.22}
\begin{tabular}{@{}cccc@{}}
\toprule
$n$ & closed area $c_n$ & closed diameter $\dclosed$ & plane diameter $\dplane$\\
\midrule
3 & $\pi/903$ & $0.0332764$ & $5.53863\!\times\!10^{-4}$\\
4 & $\pi/903$ & $0.0332764$ & $2.76893\!\times\!10^{-4}$\\
5 & $\pi/1631721$ & $7.82847\!\times\!10^{-4}$ & $1.02142\!\times\!10^{-7}$\\
6 & $\pi/1631721$ & $7.82847\!\times\!10^{-4}$ & $7.66062\!\times\!10^{-8}$\\
\bottomrule
\end{tabular}
\end{center}

The decimal values are rounded. Closed-surface diameter inequalities
are non-strict; the plane bounds are strict. None of these geometric
constants is asserted to be optimal. If a plane tiling has a finite
closed-surface quotient by embedded tiles, its lifted prototile
satisfies both estimates.

The fixed-side bounds allow reflex angles and do not assume
vertex-properness. Uniformity in the number of sides uses
vertex-properness: the constructed tilings in Zare's family have
vertices incident with only two tiles. Geodesic sides are also
essential for the plane theorem; Section~\ref{sec:curved-rectangles}
gives vertex-proper tilings by four-corner curved tiles with diameters
tending to zero.

\section{Diameter calculations for classical convex monotiles}\label{sec:classical-diameters}

We calculate the diameters of two small-area families of convex
monotiles. In each small-area limit below, the diameter has a positive
limit.

\subsection{Diameter from vertices}

For a compact convex geodesic polygon with vertices $p_1,\ldots,p_n$,
\begin{equation}\label{eq:vertex-diameter}
 \diam(P)=\max_{i,j}d(p_i,p_j).
\end{equation}
Indeed, distance to a fixed point is convex along a geodesic segment
in $\HH$. Triangulating $P$ shows that its maximum over $P$ is
attained at a vertex. Applying this first to one argument of
$d(p,q)$ and then to the other proves \eqref{eq:vertex-diameter}.
In the upper half-plane, write $p_i=\xi_i+\mathrm{i}\eta_i$, where $\eta_i>0$.
Then
\begin{equation}\label{eq:coordinate-diameter}
 \diam(P)=\arcosh\left(
 1+\max_{i,j}\frac{(\xi_i-\xi_j)^2+(\eta_i-\eta_j)^2}{2\eta_i\eta_j}
 \right).
\end{equation}
For a triangle, the diameter is simply its longest side.

\subsection{Triangular monotiles with arbitrarily small area}

Di Giosia et al.\ \cite[arXiv v1, Propositions 4.8--4.9]{digiosia}
give isosceles triangular monotiles of every area $0<a<\pi$.
For a positive integer $k$, their angle parameters are
\[
 \alpha=\frac{a}{2k-1},\qquad
 \beta=\frac\pi2-k\alpha,\qquad
 (\alpha,\beta,\beta).
\]
Choose $k$ so that the angles are positive. They satisfy
$2k\alpha+2\beta=\pi$, the tiling criterion in their Proposition 4.8,
and Gauss--Bonnet gives area $a$.

Let $L$ be either equal side, opposite $\beta$, and let $b$ be the
base, opposite $\alpha$. The angle form of the hyperbolic cosine law
\cite[arXiv v1, Proposition 3.2]{digiosia} gives
\begin{align}
 \cosh L
 &=\frac{\cos\beta+\cos\alpha\cos\beta}{\sin\alpha\sin\beta}
   =\cot(\alpha/2)\tan(k\alpha),\label{eq:triangle-leg}\\
 \cosh b
 &=\frac{\cos\alpha+\cos^2\beta}{\sin^2\beta}
   =\frac{\cos\alpha+\sin^2(k\alpha)}{\cos^2(k\alpha)}.
   \label{eq:triangle-base}
\end{align}
Thus the exact diameter is $\max\{L,b\}$.
For fixed $k$ and $a\downarrow0$, we have
\[
 \cosh L\longrightarrow2k,\qquad
 \cosh b\longrightarrow1,
 \qquad \diam(T_a)\longrightarrow\arcosh(2k)>0.
\]
The first limit follows from $\cot(\alpha/2)\sim2/\alpha$ and
$\tan(k\alpha)\sim k\alpha$.

In particular, take $k=1$ and $0<a<\pi/4$. Then $\alpha=a<\beta$,
so the equal sides are the longest. Formula~\eqref{eq:triangle-leg}
simplifies to
\begin{equation}\label{eq:small-triangle-diameter}
 \diam(T_a)=\arcosh(1+\sec a)
 \longrightarrow\arcosh2=\log(2+\sqrt3).
\end{equation}

The aperiodic triangular and rhombic family of Margulis--Mozes
described by Benoist~\cite[Proposition 4.8]{benoist} is different
from the small-area $n$-gons of \cite[Theorem 5]{margulis}.
Its triangle has angles $(\alpha,\beta,\beta)$ with
$5\alpha+8\beta=2\pi$ and $\alpha/\pi\notin\mathbb Q$,
and hence area $\pi/2+\alpha/4>\pi/2$. It has no small-area degeneration.
Theorem~\ref{thm:convex-diameter} nevertheless applies to all the
convex families, without requiring their construction parameters.

\subsection{A small-area geodesic horocyclic family}

The hyperbolic Penrose construction underlying another branch of
the Margulis--Mozes examples is described explicitly by
Oyono-Oyono and Petite \cite[Section 2.1]{oyonopetite}.
We introduce a width parameter and compute the resulting diameter.

Fix an integer $m\ge2$ and $w>0$. In the upper half-plane, join in
order the points
\[
 0+i,\ \frac wm+i,\ldots,w+i,\ w+mi,\ mi
\]
by geodesic segments, and call the polygon $P_{m,w}$.
It has $m+3$ genuine sides. The case $(m,w)=(2,1)$ is the
pentagon in \cite[Section 2.1]{oyonopetite}.

To check the tiling directly, let $\varphi$ be the periodic graph formed
by geodesic chords between successive points $jw/m+i$.
On the interval $[jw/m,(j+1)w/m]$,
\[
 \varphi(x)^2=1+(x-jw/m)((j+1)w/m-x).
\]
The lower boundary of $P_{m,w}$ is $y=\varphi(x)$ for $0\le x\le w$;
the upper boundary is $y=m\varphi(x/m)$, the single chord from $mi$ to
$w+mi$. The upper graph lies strictly above the lower one, since
on each such interval
\[
 m^2+x(w-x)>
 1+(x-jw/m)((j+1)w/m-x).
\]
The copies $m^j(P_{m,w}+kw)$, $j,k\in\ZZ$, tile the bands between
$m^j \varphi(x/m^j)$ and $m^{j+1}\varphi(x/m^{j+1})$.
Adjacent levels have identical boundary graphs. As $j\to-\infty$
and $j\to\infty$, these graphs tend uniformly to $0$ and to
infinity, respectively, because $1\le \varphi\le\sqrt{1+(w/2m)^2}$.
The copies therefore tile all of $\HH$ and are locally finite.

Put $\eta=\arctan(w/(2m))$. Conformality of the model gives
two lower endpoint angles $\pi/2-\eta$, two upper angles
$\pi/2+\eta$, and $m-1$ intermediate lower angles $\pi-2\eta$.
They all lie in $(0,\pi)$, so the simple geodesic polygon is convex.
Gauss--Bonnet gives
\begin{equation}\label{eq:penrose-area}
 \Area(P_{m,w})=2(m-1)\arctan\frac w{2m}.
\end{equation}

All lower vertices have height $1$ and the two upper vertices have
height $m$. By \eqref{eq:coordinate-diameter}, the largest distance
is either between the two lower endpoints or between opposite
lower and upper endpoints. Consequently
\begin{equation}\label{eq:penrose-diameter}
 \diam(P_{m,w})=
 \arcosh\max\left\{1+\frac{w^2}{2},\frac{m^2+1+w^2}{2m}\right\}.
\end{equation}
For fixed $m$ and $w\downarrow0$,
\[
 \Area(P_{m,w})\longrightarrow0,\qquad
 \diam(P_{m,w})\longrightarrow
 \arcosh\frac{m+m^{-1}}2=\log m>0.
\]
In particular, the pentagonal family has limiting diameter $\log2$.

\section{Tiles of small area and diameter}\label{sec:small}

Agol \cite{agol} described a binary tiling with arbitrarily small tile
area, referring to B\"or\"oczky's construction \cite{boroczky}.
Zare \cite{zare} explained how to obtain arbitrarily small diameter
by using $q$ arcs on one boundary and $q+1$ on the other; he had
already mentioned the small-diameter phenomenon in 2010
\cite{zare2010}. We give an explicit geodesic realisation of this idea.

\subsection{Binary tilings}

In the upper half-plane model, let
$B_\varepsilon=[0,\varepsilon]\times[1,2]$ for $\varepsilon>0$.
The copies $B_{j,k}=2^j(B_\varepsilon+k\varepsilon)$, $j,k\in\ZZ$,
are congruent in the hyperbolic metric. For each $j$, they fill the
strip $2^j\le y\le2^{j+1}$, and these strips tile $\HH$. The lower
boundary of each tile meets two tiles in the row below, while its
upper boundary meets one tile in the row above. Subdividing the
lower boundary at its midpoint gives the five combinatorial sides
of the binary tile: three horocyclic arcs and two geodesic segments.
This is the horocyclic form of the B\"or\"oczky tiling; see also
Dolbilin and Frettl\"oh \cite{dolbilinfrettloh}.

The area is
\begin{equation}\label{eq:binary-area}
 \Area(B_\varepsilon)
 =\int_0^\varepsilon\int_1^2\frac{dy\,dx}{y^2}
 =\frac\varepsilon2\longrightarrow0.
\end{equation}
Its diameter is at least the distance between $(0,1)$ and $(0,2)$,
namely $\log2$. Thus this family has arbitrarily small area, while
its diameter does not tend to zero.

\subsection{A geodesic realisation of Zare's idea}

\begin{theorem}\label{thm:small}
For each integer $q\ge2$, a compact simple geodesic $(2q+3)$-gon $T_q$ admits a
locally finite tiling of $\HH$ by congruent copies, with
\[
 \Area(T_q)=2\arctan\frac1{2q^2}<\frac1{q^2},
 \qquad \diam_{\HH}(T_q)<\frac3q.
\]
The tiling is not isohedral and has no compact quotient.
\end{theorem}

\begin{proof}
Use the upper half-plane metric $ds^2=(dx^2+dy^2)/y^2$. Put
\begin{equation}\label{eq:parameters}
 r=\frac{q+1}{q},\qquad \delta=\frac1{q^2},\qquad
 L=(q+1)\delta,\qquad H=\sqrt{1+\frac{\delta^2}{4}}.
\end{equation}
Define the continuous $\delta$-periodic function
\begin{equation}\label{eq:scallops}
 \varphi(x)=\sqrt{H^2-\left(x-\left(k+\tfrac12\right)\delta\right)^2},
 \qquad k\delta\le x\le(k+1)\delta,\quad k\in\ZZ.
\end{equation}
Each graph segment is the geodesic joining $(k\delta,1)$ and
$((k+1)\delta,1)$: it lies on a circle centred on the real axis.
Since $1\le \varphi\le H<r$, the region
\begin{equation}\label{eq:tile}
 T_q=\{(x,y):0\le x\le L,\quad \varphi(x)\le y\le r \varphi(x/r)\}
\end{equation}
is a simple geodesic polygon. It has $q+1$ lower arcs, $q$ upper arcs
and two vertical sides. Its $q-1$ internal upper corners are reflex:
with $\beta=\arctan(\delta/2)$, their interior angles are exactly
\[
 \pi+2\beta=\pi+2\arctan\frac1{2q^2}>\pi.
\]
The internal lower angles are $\pi-2\beta$, and the four end angles
are $\pi/2\pm\beta$. Thus every $T_q$, $q\ge2$, is nonconvex,
as expected from the convex special case of
Theorem~\ref{thm:convex-diameter}.

For $j,m\in\ZZ$, let
\begin{equation}\label{eq:copies}
 T_{j,m}=r^j(T_q+mL),\qquad C_j(x)=r^j \varphi(x/r^j).
\end{equation}
The copies are congruent because translation and dilation are
hyperbolic isometries. Since $L/\delta=q+1$ and $L/(r\delta)=q$ are
integers, periodicity gives
\begin{equation}\label{eq:row}
 T_{j,m}=\{(x,y):mLr^j\le x\le(m+1)Lr^j,
                       \ C_j(x)\le y\le C_{j+1}(x)\}.
\end{equation}
Thus row $j$ fills the strip between $C_j$ and $C_{j+1}$. Moreover,
\[
 r^j\le C_j(x)\le Hr^j<r^{j+1}\le C_{j+1}(x).
\]
The graphs tend uniformly to zero as $j\to-\infty$ and to infinity
as $j\to+\infty$. Hence the strips cover $\HH$ with disjoint interiors.
A compact set meets finitely many rows and, because its horizontal
coordinate is bounded, finitely many tiles in each row. The tiling
is locally finite.

This tiling is not vertex-proper. Indeed,
$(\delta,1)$ is a genuine corner on $C_0$, but is not an endpoint
of a vertical side in either adjacent row. The vertical sides in
row $0$ occur at multiples of $(q+1)\delta$, and those in row $-1$
at multiples of $q\delta$. Since $q\ge2$, exactly two tiles meet
at $(\delta,1)$; their angles are $\pi-2\beta$ and $\pi+2\beta$.

Let $I=\int_0^\delta dx/\varphi(x)=2\arctan(\delta/2)$. By periodicity and
the substitution $u=x/r$,
\begin{align*}
 \Area(T_q)
 &=\int_0^L\left(\frac1{\varphi(x)}-\frac1{r \varphi(x/r)}\right)dx\\
 &=(q+1)I-qI
  =2\arctan\frac1{2q^2}<\frac1{q^2}.
\end{align*}
Also $T_q\subset[0,L]\times[1,rH]$. From the lower of two points,
move vertically to the height of the other, then horizontally.
This path bounds their
ambient distance, whether or not it lies in the tile. Since
$\log r<1/q$ and $\log H\le1/(8q^4)$,
\begin{equation}\label{eq:smalldiam}
 \diam_{\HH}(T_q)\le L+\log(rH)
 <\frac2q+\frac1{q^2}+\frac1{8q^4}<\frac3q.
\end{equation}

Each arc edge has length $2\operatorname{arsinh}(\delta/2)<\delta$,
whereas the two vertical sides have length
$\log r>1/(q+1)\ge\delta$. Their supporting geodesics meet at the
ideal point $\infty$. Thus any symmetry of the tiling fixes $\infty$
and has the form
\[
 x+iy\longmapsto\pm ax+b+iay,\qquad a>0.
\]
The vertical sides have endpoint heights $r^j,r^{j+1}$, which force
$a=r^k$ for some $k\in\ZZ$. Dilation by $r$ preserves the tiling.
Composing with $z\mapsto r^{-k}z$, we may assume $a=1$. Preservation
of the vertical-side grid in every row then requires
\[
 b\in\bigcap_{j\in\ZZ}Lr^j\ZZ=\{0\}.
\]
Since $\varphi(-x)=\varphi(x)$, reflection in the imaginary axis is also a
symmetry. The full group is therefore
\begin{equation}\label{eq:group}
 G=\{z\mapsto r^kz,\quad z\mapsto-r^k\overline z:k\in\ZZ\}.
\end{equation}
It sends $(j,m)$ to $(j+k,m)$ or $(j+k,-m-1)$, so it is not
transitive on tiles. In polar coordinates, the quotient by its
orientation-preserving subgroup has area
\[
 \int_1^r\frac{d\rho}{\rho}
 \int_0^\pi\frac{d\theta}{\sin^2\theta}=\infty.
\]
The full group also has infinite covolume. Hence no subgroup of $G$
is a cocompact surface group, and this tiling has no compact quotient.
\end{proof}

The maps $a(z)=z+\delta$ and $b(z)=rz$ satisfy
$b^{-1}a^{q+1}b=a^q$, the Baumslag--Solitar relation in Zare's
description. The translation $a$ is not a symmetry of the tiling.

\section{Extensions}\label{sec:extensions}

We consider regular curved sides on closed surfaces, the additional
consequences of tile-transitive symmetry, and small curved tiles in
the entire plane.

\subsection{Curvilinear tiles on closed hyperbolic surfaces}\label{sec:curved}

The fixed-corner bounds extend to curved sides because the signed
boundary-curvature integral of each tile in a finite congruent tiling
vanishes.

\subsubsection{Regularity and boundary curvature}

Let $T_1,\ldots,T_f$ be congruent embedded closed discs tiling a closed
hyperbolic surface $S$, with $\chi=\chi(S)<0$. Assume that every boundary consists
of finitely many regular $C^2$ arcs and that the boundaries admit a
finite common subdivision. Each open subdivision edge must belong to
exactly two tile boundaries. Congruence refers to the lifted shapes
in $\HH$ and may reverse orientation.

A genuine corner has angle
$\ang{i}\in(0,2\pi)\setminus\{\pi\}$. The number $n$ counts these
corners, excluding auxiliary smooth subdivision points. Reflex angles
and corners lying on the smooth part of another boundary are allowed.
At such a point the smooth incidence contributes angle $\pi$,
so the same full- and half-vertex conventions apply.

For a tile $P$, define
\[
 J(P)=\int_{\partial P_{\rm sm}}k_g\,ds,
 \qquad k_g=\langle\nabla_\tau\tau,\nu_{\rm in}\rangle,
\]
where $\tau$ is a unit tangent and $\nu_{\rm in}$ is the inward unit
conormal in $S$. Thus $k_g$ is signed geodesic curvature; its sign
depends on the side of the arc occupied by the tile. This definition
does not require a global orientation of $S$.

\begin{lemma}[Boundary-curvature cancellation]\label{lem:curved-cancel}
Every tile in this finite congruent tiling satisfies $J(T_i)=0$.
If its area is $a$, then
\begin{equation}\label{eq:curved-GB}
 a=(n-2)\pi-\sum_{i=1}^n\ang{i}.
\end{equation}
In particular, $n\ge3$.
\end{lemma}
\begin{proof}
Congruence preserves $J(P)$, including under orientation reversal:
an isometry preserves the inner product in the definition of $k_g$,
and replacing $\tau$ by $-\tau$ leaves $\nabla_\tau\tau$ unchanged.
Hence all tiles have a common value $J$.

Along any shared open boundary edge, the two tiles have the same
tangent acceleration and opposite inward conormals. Their signed
curvatures therefore add to zero. Integrating along all edges in the
finite common subdivision gives
\[
 \sum_{i=1}^f J(T_i)=0.
\]
Since the left side is $fJ$, we have $J=0$.

Gauss--Bonnet for a disc with corners in curvature $-1$ reads
\[
 -a+J(P)+\sum_{i=1}^n(\pi-\ang{i})=2\pi.
\]
Substituting $J(P)=0$ proves \eqref{eq:curved-GB}.
If $n\le2$, its right side is nonpositive, contradicting $a>0$.
\end{proof}

\Needspace{15\baselineskip}
\subsubsection{Area, diameter, and tile-count bounds}

\begin{theorem}\label{thm:curved}
For the regular curvilinear tilings above, all fixed-corner area,
diameter, and tile-count bounds in Section~\ref{sec:explicit} hold:
\[
 a\ge\pi \varepsilon_n,\qquad
 \diam_S(T_i)\ge\arcosh(1+\varepsilon_n/2),\qquad
 f\le-2\chi B_{\lceil n/2\rceil+1}.
\]
If the tiling is vertex-proper, then the uniform bounds
\eqref{eq:valence-explicit} and $f\le-3263442\chi(S)$ also hold.
For curved sides, Definition~\ref{def:vertex-proper} uses the genuine
corners specified above. Equivalently, full vertices have at least
three genuine corner incidences and half vertices at least two.
Such a tiling is automatically vertex-proper when every genuine
angle is less than $\pi$.
\end{theorem}
\begin{proof}
At a genuine tiling vertex, the tangent sectors of incident tiles
partition the tangent plane. At most one incident tile can be smooth
there, since each such incidence contributes $\pi$ and a positive
corner angle is present. Thus the local angle equations remain
\[
 \sum_i m_i(w)\ang{i}+h(w)\pi=2\pi,
 \qquad h(w)\in\{0,1\}.
\]
Each labelled corner occurs once per tile, so
$\sum_w m_i(w)=f$. Lemma~\ref{lem:curved-cancel} supplies the same
area--angle formula used in \eqref{eq:GB}. The polytope construction
and integer-gap proof therefore apply without change.

The diameter and tile-count estimates use only area and the ambient
hyperbolic metric. For the uniform bound, suppress smooth auxiliary
vertices of degree two in the common subdivision. With $v_{\mathrm{full}}$ full
vertices and $v_{\mathrm{half}}$ half vertices, the counts are again
\[
 nf\ge3v_{\mathrm{full}}+2v_{\mathrm{half}},\qquad 2e=nf+v_{\mathrm{half}}.
\]
The proof of Proposition~\ref{prop:vertices} gives
$a\ge\pi(n-6+v_{\mathrm{half}}/f)/3$, and Corollary~\ref{cor:unrestricted} applies.
Finally, angles less than $\pi$ require at least three corner
incidences to sum to $2\pi$ at a full vertex, and at least two to
sum to $\pi$ at a half vertex.
\end{proof}

The identity $J(P)=0$ allows positive and negative curvature on
different boundary arcs, so the arcs need not be geodesic. For a
geodesically convex tile, $k_g\ge0$ on each smooth arc; continuity and
$J(P)=0$ then force every arc to be geodesic. Corner angles less than
$\pi$ alone do not imply geodesic convexity of a curvilinear tile.

\subsection{Symmetry and compact quotients}\label{sec:transitive}

Siegel's covolume bounds~\cite{siegel}, in curvature $-1$, are
\begin{equation}\label{eq:siegel}
\begin{aligned}
 \Area(\Gamma\backslash\HH)&\ge\pi/21,
 &&\Gamma\le\operatorname{Isom}^{+}(\HH),\\
 \Area(\Gamma\backslash\HH)&\ge\pi/42,
 &&\Gamma\le\operatorname{Isom}(\HH),
\end{aligned}
\end{equation}
for cofinite discrete groups. Equality is attained by the
orientation-preserving $(2,3,7)$ triangle group and the reflection
group, respectively; see also Beardon~\cite{beardon} and
Macbeath~\cite{macbeath}.

A group $G$ preserving a locally finite tiling by compact geodesic
polygons is discrete. Indeed, if $g_j\to1$, the interiors of a fixed
tile $P$ and $g_jP$ meet for large $j$, so $g_jP=P$. The stabilizer
$G_P$ is finite, since it permutes the corners and three noncollinear
corners determine an isometry. Thus $g_j=1$ eventually.
If $G$ is transitive on tiles, a fundamental region for $G_P$ in $P$
is a fundamental region for $G$ up to a null set. Consequently
\begin{equation}\label{eq:isohedral}
 \Area(P)=|G_P|\,\Area(G\backslash\HH)\ge\pi/42.
\end{equation}
The bound is $\pi/21$ if an orientation-preserving group is transitive.

Similarly, suppose a group $G$ of global isometries of a closed
hyperbolic surface $S$ is transitive on its tiles, allowing curved
sides. The quotient orbifold has area $-2\pi\chi/|G|$, so
$|G|\le-84\chi$. Since $f\le|G|$ and $fa=-2\pi\chi$,
\[
 a\ge\pi/42,\qquad \diam_S(T_i)\ge\arcosh(1+1/84).
\]
For orientable $S$ and orientation-preserving $G$, these improve to
$|G|\le-42\chi$, $a\ge\pi/21$ and
$\diam_S(T_i)\ge\arcosh(1+1/42)$. When $\chi=2-2g$, the group
bounds are $168(g-1)$ and $84(g-1)$, respectively, the latter being
Hurwitz's bound. These conclusions require transitivity by global
isometries; congruence of lifted prototiles alone is insufficient.
The bounds of Section~\ref{sec:gap} instead use the exact corner
count and require no such symmetry.

\subsection{Small curved rectangles in the hyperbolic plane}\label{sec:curved-rectangles}
In the upper half-plane, fix $\varepsilon>0$, put $r=1+\varepsilon$, and let
\[
 R_\varepsilon=[0,\varepsilon]\times[1,r],\qquad
 R_{j,k}=r^j(R_\varepsilon+k\varepsilon),\quad j,k\in\ZZ.
\]
For each $j$ these congruent tiles fill the strip $r^j\le y\le r^{j+1}$;
the strips cover $\HH$ and the tiling is locally finite. The horizontal
sides are horocyclic, the vertical sides geodesic, and all four angles
are right angles. Direct integration and a vertical-horizontal path give
\[
 \Area(R_\varepsilon)=\varepsilon(1-r^{-1})
 =\frac{\varepsilon^2}{1+\varepsilon},\qquad
 \diam(R_\varepsilon)\le\varepsilon+\log(1+\varepsilon).
\]
Both tend to zero. Nevertheless, this tiling is vertex-proper.
At a vertex on $y=r^j$, the grids
of vertical sides in the two adjacent rows are
$\varepsilon r^j\ZZ$ and $\varepsilon r^{j-1}\ZZ$.
A point in exactly one grid belongs to two tiles on that side and
one on the other, hence to three distinct tiles. A point in both
grids belongs to four. These exhaust the genuine tiling vertices.
Consequently vertex-properness alone does not extend the plane
diameter theorem to arbitrary curved sides. The closed-surface
extension remains valid because of boundary-curvature cancellation.

Gauss--Bonnet gives
$J(R_\varepsilon)=\Area(R_\varepsilon)>0$, since the four exterior
angles sum to $2\pi$. In fact, the tiling admits no compact quotient
by a discrete group of symmetries. Otherwise pass to a torsion-free
orientation-preserving subgroup of finite index~\cite{beardon},
obtaining a closed hyperbolic surface. Compactness and local finiteness
give finitely many tile orbits, say $m$. Each tile has trivial
stabilizer: its stabilizer in the isometry group is compact, and its
intersection with a discrete torsion-free group is trivial. Choose
one tile from each orbit and sum its full smooth-boundary curvature
integral. Count boundary incidences with multiplicity, including
edges whose two sides project to the same tile. Every shared edge
then occurs twice with opposite inward conormals, so the sum is
zero. But all $m$ integrals equal $J(R_\varepsilon)>0$, a
contradiction. This argument does not require the projected closed
tiles to be embedded discs.

For geodesic sides, the construction in Section~\ref{sec:small}
gives arbitrarily small area and diameter with increasing corner count.
That construction is adapted from Zare~\cite{zare} and also has no
compact quotient.

This distinction also occurs in the classical construction of
B\"or\"oczky. For the edge-to-edge B\"or\"oczky tilings defined in their
Section~2, Dolbilin and Frettl\"oh~\cite[Theorem~4.4]{dolbilinfrettloh}
prove that the symmetry group is not cocompact. Such a tiling has no
closed-surface quotient, since its deck group would already act
cocompactly. Some members nevertheless have an infinite cyclic
symmetry group.

\Needspace{10\baselineskip}
\section*{Acknowledgements}

The authors thank Zhejiang Normal University for funding ZJNU Tiling
Studio, and the over thirty undergraduate and graduate students
whose sustained participation has supported the group's work on
tilings of the sphere, the plane, tori, hyperbolic surfaces and the
integers. The corresponding author thanks Terry Tao for citing the
group's classification work in his MathOverflow answer on small
congruent spherical cells~\cite{Tao}, which prompted the present
hyperbolic questions.

SageMath, Maple, and language models including DeepSeek, Doubao and
ChatGPT 6 were used for symbolic and numerical computations,
literature searches and experimental checks. The authors take full
responsibility for the mathematical content of this work.

\end{document}